\documentclass[11pt]{amsart}

\pdfoutput=1

\usepackage[T1]{fontenc}
\usepackage[a4paper,top=3.6cm,bottom=3.7cm,inner=2.3cm,outer=2.3cm]{geometry}
\usepackage[pdftex]{graphicx,graphics}
\usepackage{color,ifpdf}
\usepackage{amsmath,amssymb,amsfonts,dsfont,enumerate,url,mathrsfs}
\usepackage[pdftex]{hyperref}
 \usepackage{epsfig}

\usepackage{graphicx,tikz} 
\usetikzlibrary{shapes.symbols}
 \usetikzlibrary{shapes.geometric}
 \usetikzlibrary{shapes.callouts}
 \usetikzlibrary{shapes.misc}
 \usetikzlibrary{shapes.multipart}
  \usetikzlibrary{shapes.arrows}

\newtheorem{theorem}{Theorem}[section]
\newtheorem{proposition}[theorem]{Proposition}

\newtheorem{remark}[theorem]{Remark}
\newtheorem{lemma}[theorem]{Lemma}
\newtheorem{definition}[theorem]{Definition}
\newtheorem{corollary}[theorem]{Corollary}

\newcommand\ep{{\varepsilon}}

\newcommand\de{\delta}
\newcommand\zu{[0,1]}

\newcommand\mk{\medskip}
\newcommand\sk{\smallskip}

\newcommand{\R}{\mathbb{R}}
\newcommand{\N}{\mathbb{N}}
\newcommand{\Z}{\mathbb{Z}}

\newcommand \sjeta{\mathcal{S}_j(\eta)}
\newcommand\thetar{\theta_{\varrho}}

\newcommand\supp{\mathrm{supp}}
\newcommand\si{\sigma}
\newcommand\wM{\widetilde \M}
\newcommand\ro{\varrho}
\newcommand\Mr{ \M_\varrho}
\newcommand\wMr{\widetilde \Mr}
\newcommand\dimi{\underline \dim}
\newcommand\dims{\overline \dim}

\newcommand{\M}{\mathsf{M}}

\newcommand\reta{{\varrho\eta}}

\newcommand\lf{\lfloor}
\newcommand\rf{\rfloor}

\newcommand{\locloc}{{}}

\begin{document}
\title[Multifractal capacities]{Sparse sampling and dilation operations on a Gibbs weighted tree, and multifractal formalism}

\author{Julien Barral \and St\'ephane Seuret}         

\address{Julien Barral,  Laboratoire d'Analyse, G\'eom\'etrie et Applications, CNRS (UMR 7539), Universit\'e Sorbonne Paris Nord,  F-93430 Villetaneuse, France}
\email{barral@math.univ-paris13.fr} 
\address{St\'ephane Seuret,   Univ Paris Est Creteil, Univ Gustave Eiffel, CNRS, LAMA UMR8050, F-94010 Creteil, France}
\email{seuret@u-pec.fr}

%
 
%
%
\begin{abstract}
In this article, starting from a Gibbs capacity, we build a new random capacity by applying two simple operators, the first one introducing some redundancy and the second one  performing a random sampling.  Depending on the values of the two parameters ruling the redundancy and the sampling, the new capacity has very different multifractal behaviors. In particular,  the multifractal spectrum of the capacity may contain two to four phase transitions, and the multifractal formalism may hold only on a strict subset (sometimes, reduced to a single point) of the spectrum's domain.
\end{abstract}

\maketitle

\section{Introduction}

Multifractal analysis of capacities defined on $\R^d$ ($d\ge 1)$ is a natural  generalisation of multifractal analysis of Radon measures. Our purpose in this article is to perform the multifractal analysis of a family of random capacities parametrised by two indices $\varrho$ and $\eta$, $\varrho $ being a redundancy index and $\eta$ a lacunarity index.  These new families of multifractal objects  are  easy to build, and have a versatile multifractal structure, exhibiting many and diverse phase transitions.

\sk 

Let us  now be more specific. We  consider the set $\mathrm{Cap}([0,1]^d)$ of capacities on $[0,1]^d$, i.e. non-decreasing non-negative mappings $\mu:  \mathcal{B}([0,1]^d) \to \R^+$, where  $\mathcal{B}([0,1]^d)$ stands for the Borel subsets of $[0,1]^d$.  One way to produce such capacities is, starting from   a sequence  $c=(c_I)_{I\in\mathcal D} $ of complex numbers indexed by the set $\mathcal D$ of dyadic subcubes of $[0,1]^d$, to impose some hierarchical structure to $c$ as follows.

\begin{definition}
\label{defM}
For every $c=(c_I)_{I\in\mathcal D} $, the capacity   $\M(c) \in \mathrm{Cap}([0,1]^d)$ is defined by 
$$
\M(c):E\in\mathcal B([0,1]^d)\mapsto \sup\{|c_J|:\, J\in\mathcal D,\, J\subset E\}.
$$ 
\end{definition}

For instance, when $c=(c_I)_{I\in\mathcal D}$ is the collection of wavelet coefficients of a  H\"older continuous function $f:[0,1]\to\R$ (see Section \ref{LWS} for details),  the sequence $(\M(c)(3I))_{I\in\mathcal D}$   where  $3I$ is the union of the cube $I$ and all of its neighboring cubes of same generation,  coincides with the so-called wavelet leaders of $f$. Then,  performing the multifractal analysis of the function $f$ amounts to perform the multifractal analysis of the capacity $(\M(c)(3I))_{I\in\mathcal D}$. Actually,  the model of multifractal capacities considered in this paper, of which the model studied by the authors in \cite{BS2020} is a particular case, was motivated by the   lacunary wavelet series introduced by Jaffard in \cite{JAFF_lac}, see Section \ref{LWS}.

\sk 

In \cite{BS2020}, we defined the following random sparse sampling  on capacities.  A probability space $(\Omega, \mathcal{F},\mathbb{P})$ is fixed for the rest of the paper.
For $j\in\N$, denote by $\mathcal D_j$ the set of dyadic cubes of generation $j$  included in $[0,1]^d$.  We denote by $g(I)$ is the generation of a dyadic cube $I\subset [0,1]^d$, i.e. $g(I)=j$ whenever $I\in \mathcal{D}_j$. By definition, $\mathcal{D}=\bigcup_{j\geq 0} \mathcal{D}_j$.

\begin{definition}
\label{defMeta}
Let $\eta\in (0,1)$ and let  $(p_I)_{I\in\mathcal D}$  be a family of independent random variables such that each $p_I$ is a Bernoulli variable of parameter $2^{-g(I)d(1-\eta)}$. For  every $\mu\in \mathrm{Cap}([0,1]^d)$, one introduces  the sequence $(S_\eta\mu)$ defined as 
$$
(S_\eta\mu)= (S_\eta\mu(I))_{I\in\mathcal D}=(p_I\,\mu(I))_{I\in\mathcal D}.
$$
Then, the random capacity $\M_\eta(\mu)$ is defined as
$$
\M_\eta(\mu):=\M(S_\eta\mu).
$$ 
\end{definition}

To get  $\M_\eta(\mu)$,  one first applies a random sparse sampling to the sequence $(S_\eta\mu)$. Observe that as $j\to\infty$, this sampling operation keeps asymptotically about $  2^{jd\eta}$ among the $2^{jd}$ values $\mu(I)$ for  $I\in \mathcal{D}_j$, and it puts the other values to 0.  Then, from this sparse sequence one gets $\M_\eta (\mu)$ by re-imposing a  hierarchical structure via the operator $\M $ defined above.

When   $\mu(E) =\lambda_\gamma(E):=(\mathcal L^d (E))^\gamma$, where $\mathcal L^d $ stands for the Lebesgue measure and $\gamma>0$, the capacity $\M_\eta(\lambda_\gamma)$ is very similar to the wavelet leaders associated to the lacunary wavelet series considered by Jaffard in~\cite{JAFF_lac}.  While $\lambda_\gamma$ is essentially the simplest monofractal object (in a sense to be precised below) fully supported over $[0,1]^d$, the capacity $\M_\eta(\lambda_\gamma)$ is a rich multifractal object. In particular, it obeys the multifractal formalism in the following sense.

\sk

The topological support of    $\mu\in \mathrm{Cap}([0,1]^d)$  is defined as 
$$
\displaystyle \mathrm{supp}(\mu)=\{x\in[0,1]^d: \, \mu(B(x,r))>0,\ \forall \, r>0\}.
$$

\sk 
If  $\supp(\mu)\neq\emptyset$, set 
\begin{equation}
\label{deftau}
\tau_\mu:\,q\in\R\mapsto\liminf_{j\to\infty} \tau_{\mu,j}(q), \ \ \ \mbox{ where } \  \tau_{\mu,j}(q) :=\frac{-1}{j}\log_2\sum_{I\in \mathcal D_j, \mu(I)>0}\mu(3I)^q. 
\end{equation}
An equivalent definition is 
$$
\tau_\mu(q)=\liminf_{r\rightarrow 0^+}\frac{\log \sup
\sum_{i}\mu(B(x_i,r))^q}{\log(r)}
$$
where the supremum is taken over all  families of disjoint closed balls $%
\{B(x_i,r)\}_{i}$  of radii $r$ with centers $x_{i}\in \mbox{supp}(\mu)$, and $\mu(B(x,r)):=\mu(B(x,r)\cap[0,1]^d)$. This function is concave and non-decreasing, and it takes values in $\R\cup\{-\infty\}$. 

The {\em singularity}, or {\em multifractal}, {\em spectrum} of~$\mu$  is then defined as
$$
\sigma_\mu:H\mapsto \dim \underline E_\mu(H), \quad H\in\R,
$$
where 
$$
\underline E_\mu(H)=\left\{x\in \mathrm{supp}(\mu):\, \liminf_{r\to 0^+} \frac{\log \big(\mu(B(x,r))  \big)}{\log(r)}=H\right\},
$$
and where $\dim$ stands for the Hausdorff dimension,  and by convention, $\dim E= -\infty$ if and only if $E$ is empty. The singularity spectrum  provides a fine geometric description of $\mu$ at small scales by giving the Hausdorff dimension of the iso-H\"older sets $\underline E_\mu(H)$. Note that since~$\mu$ is a capacity, the set $\underline E_\mu(H)$ is empty when $H<0$. 
It is always true (see \cite{BRMICHPEY,Olsen,Riedi,LN} for measures and \cite{JLVVOJAK} for capacities), that 
\begin{equation}
\label{formalism-1}
\sigma_\mu(H)\le \tau_\mu^*(H),  \ \  \forall \, H\in \R,
\end{equation}
where for any function $f:\mathbb{R}\to \mathbb{R}\cup\{-\infty\}$, the Legendre transform $f^*$ of $f$ is defined as 
$$
f^*:H\in\mathbb{R}\mapsto \inf\{Hq-\tau_\mu(q):\, q\in\mathbb R\}.
$$
The domain of such a function $f$, that is the set of points where it takes real values, will be denoted by $\mathrm{dom}(f)$. 

The multifractal formalism holds for $\mu$ at $H$ when the   inequality \eqref{formalism-1} turns out to be  an equality. It holds for $\mu$ when  \eqref{formalism-1} is an equality for all $H\in\R$. 

A capacity $\mu$ is monofractal  when $\tau_\mu$ is linear or more generally if $\underline E_\mu(H_0)=\supp(\mu)$ for some $H_0\in\R_+$, and   $\mu$ is called multifractal when  there are at least two distinct non-negative exponents $H$ and $H'$ such that  $\underline E_\mu(H)\neq\emptyset\neq \underline E_\mu(H')$.

Coming back to the capacity  $\lambda_\gamma(E)=(\mathcal L^d (E))^\gamma$, for $\gamma>0$, it is clear that 
\begin{align*}
\sigma_{\lambda_\gamma}(H)
=\begin{cases} \ \ d&\text{if }H=d\gamma\\
-\infty&\text{otherwise}
\end{cases} \ \ \mbox{ and } \ \ 
\tau_{\lambda_\gamma}(q)=d(\gamma q-1), q\in\R.
\end{align*}
Consequently, $\lambda_\gamma$ is monofractal and  it obeys the multifractal formalism. The situation is more involved for    $\M_\eta (\lambda_\gamma)$: it is shown in \cite{BS2020} (see also \cite{JAFF_lac} for a slightly different model) that almost surely, %
\begin{align*}
\sigma_{\M_\eta (\lambda_\gamma)}(H) = 
\begin{cases}\frac{\eta}{\gamma} H&\text{if } H\in [d\gamma,d \gamma/\eta]\\
-\infty&\text{otherwise},
\end{cases}
\ \mbox{ and } \ 
\tau_{\M_\eta (\lambda_\gamma)}(q) =
\begin{cases}\eta^{-1}d(\gamma q-\eta)&\text{if }q<1/\gamma\\
 d(\gamma q-\eta)&\text{if }q\ge1/\gamma
 \end{cases}
\end{align*}
Thus, $\M_\eta (\lambda_\gamma)$ is multifractal, and it  also obeys the multifractal formalism.

\begin{tikzpicture}[xscale=2.5,yscale=2.5]
{\tiny
\draw [->] (0,-0.2) -- (0,1.2) [radius=0.006] node [above] {$\sigma_{\lambda_\gamma}(H)$};
\draw [->] (-0.2,0) -- (2.1,0) node [right] {$H$};
 
\draw [dotted]   (0,1) --  (0.7,1) ;
 
%
%

 \draw[fill,color=blue] (0.7,1) circle [radius=0.03];
 \draw [dashed] (0.7,1) -- (0.7,-0.0);
 \draw [fill] (0.7,-0.0) circle [radius=0.03] node [below]  {$d\gamma$};

 \draw [fill] (-0.1,-0.10)   node [left] {$0$}; 
 \draw [fill] (0,1) circle [radius=0.03] node [left] {$d \ $};

  }
\end{tikzpicture} 
   
\vskip -32.5mm 
%
%
%
%
%

\vskip -7.7mm 
\hskip 70mm     
   \begin{tikzpicture}[xscale=2.5,yscale=2.5]
{\tiny
\draw [->] (0,-0.2) -- (0,1.2) [radius=0.006] node [above] {$\sigma_{\M_\eta(\lambda_\gamma)}(H)$};
\draw [->] (-0.2,0) -- (2.9,0) node [right] {$H$};

%
%

 \draw[fill] (1,0.4) circle [radius=0.03] [dashed] (1,0.4) -- (1,-0.0) [fill] circle [radius=0.03]  
 node [below]  {$d\gamma$};
 \draw[fill] (2.5,1) circle [radius=0.03] [dashed] (2.5,1) -- (2.5,-0.0) [fill] circle [radius=0.03] 
  node [below]  {$d\gamma/\eta$};


\draw [thick,color=blue]    (1,0.4) -- (2.5,1); 
\draw [dotted,color=blue]   (0,0) --  (1,0.4) ;
\draw [dotted]   (0,0.4) --  (1,0.4) ;
\draw [dotted]   (0,1) --  (2.5,1) ;
 
 \draw [fill] (-0.1,-0.10)   node [left] {$0$}; 
 \draw [fill] (0,1) circle [radius=0.03] node [left] {$d\ $}; 
 \draw [fill] (0,.4) circle [radius=0.03] node [left] {$d\eta \ $}; 
} 
\end{tikzpicture} 
\sk 

Let us now introduce some redundancy in the constructions.

\begin{definition}
For $\varrho\in (0,1]$, $j\in \N$ and $I\in\mathcal D_j$, let $I^\varrho$ denote the   unique dyadic cube in $\mathcal D_{\lfloor \varrho j\rfloor}$  of sidelength $2^{-\lfloor \ro j\rfloor} $  containing $I$.

 The $\varrho$-dilation operation on $\mathrm{Cap}([0,1]^d)$ is defined by 
$$
D_\varrho \mu=\M(c_\varrho(\mu)), \text{ where }
c_\varrho(\mu)=(\mu(I^\varrho))_{I\in\mathcal D}.
$$
\end{definition}
Note that $D_1$ is the identity map. 

When $\ro<1$, some redundancy is introduced in $D_\varrho \mu$, since each $I\in \mathcal{D}_j$ shares the same cube $I^\varrho$ with the $2^{j-\lfloor \varrho j\rfloor}$ dyadic subcubes of $I^\varrho$ of generation $j$. In addition, observe that  especially if the measure $\mu$ is doubling, the measures $\mu(I^\varrho) $ and $\mu((I')^\varrho)$ for two neighbor cubes  $I$ and $I'$ are very comparable.

Remark  that the family $(\lambda_\gamma)_{\gamma>0}$ possesses the rather exceptional property that  $D_\varrho \lambda_\gamma$ still obeys the multifractal formalism. Indeed,  $D_\varrho \lambda_{\gamma}\approx  \lambda_{\gamma\varrho}$, where for two capacities $\mu$ and $\mu'$, $\mu\approx\mu'$ means that there exists $C\geq 1$ such that $C^{-1} \mu(I)\le \mu'(I)\le C\mu(I)$ for all $I\in \mathcal D$.   

If $\mu\in\mathrm{Cap}([0,1]^d)$ and $\supp(\mu)\neq\emptyset$,  one checks that 
$$
 \sigma_{D_\varrho \mu}(H)=\sigma_\mu(H/\varrho),\forall \, H\in\R
 $$
  and that if $\supp(\mu)=[0,1]^d$, 
  $$\tau_{D_\varrho \mu}(q)=\varrho \tau_\mu(q)+d(\varrho-1), \ \forall\, q\in\R,
$$
whose Legendre transform equals $\varrho \tau_\mu^*(\cdot/\varrho)+d(1-\varrho)$. In particular, if $\mu$ obeys the multifractal formalism, then $\sigma_{D_\varrho \mu}=\tau_\mu^*(\cdot/\varrho)$, so  that ${D_\varrho \mu}$ obeys the multifractal formalism if and only if  $\tau_\mu^*(\cdot/\varrho)= \tau_{D_\varrho \mu}^*=\varrho \tau_\mu^*(\cdot/\varrho)+d(1-\varrho)$, i.e. $\tau_\mu^*$ takes its values in $\{d,-\infty\}$ if $\varrho<1$. Assuming that $\tau_\mu'(0)$  exists, and setting $\gamma=\tau_\mu'(0)/d$,  this equality implies that necessarily $\tau_\mu=\tau_{\lambda_\gamma}$ and $\sigma_\mu=\sigma_{\lambda_\gamma}$. So, as soon as $\mu$ is multifractal and $\tau_\mu$ is differentiable at 0, the capacity $D_\varrho \mu$ for  $\varrho<1$ cannot obey the multifractal formalism when $\mu$ does. 

\sk 

Let us introduce a notation (that we already used to define $\lambda_\gamma = (\mathcal{L}_d)^\gamma$):  for any $\gamma>0$ and $\mu\in \mathrm{Cap}([0,1]^d)$, the capacity $\mu^\gamma$ is defined by  
\begin{equation}
\label{defpower}
 \mu^\gamma  :E\in\mathcal B([0,1]^d)\mapsto\mu(E)^\gamma.
\end{equation}

Using    the previous observations, one sees that for any $\mu\in \mathrm{Cap}([0,1]^d)$, $\gamma>0$ and $\varrho\in (0,1]$,  that $ (D_\varrho(\mu))^\gamma =D_\varrho  (\mu^\gamma)$, and
$$
\sigma_{D_\varrho  (\mu^\gamma)}=\sigma_\mu(\cdot/\gamma\varrho),\quad \tau_{D_\varrho \mu}(q)=\varrho \tau_\mu(\gamma\cdot)+d(\varrho-1), \quad \text{and} \quad \sigma_{D_\varrho(\mu)}=\sigma_{(\mu^\varrho)}.
$$
It is clear that $\mu^\gamma$ satisfies the multifractal formalism as soon as $\mu$ does.   Also, if $\mu \in\mathrm{Cap}([0,1]^d)$  satisfies the   three properties $(i)$ $\mu$ obeys the multifractal formalism, $(ii)$ $\mu$ is multifractal, and $(iii)$ $\tau_\mu$ is differentiable at 0, then the capacities $D_\varrho  (\mu^{1/\varrho})$,  $0<\varrho<1$,  share the same singularity spectrum as $\mu$ but none of them obeys the multifractal formalism. 

\sk 

In this article, we are interested in combining   $D_\varrho$ and $\M _\eta$. 
\begin{definition}
\label{defMetarho}
For every $\mu \in \mathrm{Cap}([0,1]^d)$, for every $\eta \in (0,1]$ and $0<\ro \leq 1/\eta$,   the capacity $\M_{\ro,\eta}(\mu)$ is defined by
$$
\M_{\ro,\eta}(\mu)=\M_\eta (D_{\varrho\eta} \mu)=\M(S_\eta (D_{\varrho\eta} \mu)).
$$
\end{definition}

 The  two operators are somehow complementary: one one side, $D_{\varrho\eta}$ creates redundancy since $\varrho\eta$ is always less than 1, while on the other side  $\M _\eta$ introduces lacunarity. 
 
Our goal is to study  the multifractal nature of  the elements of $\{M_{\ro,\eta}(\mu): 0<\varrho\le 1/\eta\}$ when $\mu$ is  a \textit{Gibbs capacity}, defined as follows.

Let  $\varphi:\R^d\to \R$ be $\Z^d$-invariant  H\"older continuous potential.  It is standard that 
the sequence of Radon measures 
$$
\nu_n(\mathrm{d}x)= \frac{\displaystyle\exp\left (S_n\varphi(x)\right)}{\int_{[0,1]^d}  \exp\left (S_n\varphi(t)\right)\mathcal{L}^d(\mathrm{d}t)}\, \mathcal{L}^d(\mathrm{d}x), \quad \text{where }S_n\varphi(x)=\sum_{k=0}^{n-1} \varphi(2^n x),
$$
converges vaguely to a measure fully supported on $\R^d$ (see \cite{Bowen}). Denote by  $\nu$  the restriction of this \textit{Gibbs} measure to~$[0,1]^d$. 
\begin{definition}
\label{defgibbscap}
A Gibbs capacity $\mu$ is a capacity $\mu\in\mathrm{Cap}([0,1])^d$ such that  $\mu\approx \nu^\gamma$, where    $\gamma>0$  and $\nu$ is such a Gibbs measure on $\zu^d$ as above. 
\end{definition}

Classical results on Gibbs measures (see \cite{Ruelle,ColletLebPor,Olsen}) give :
\begin{itemize}
\item One has
$$
\tau_{\nu}(q)=\big (qP(\varphi)-P(q\varphi)\big )/\log(2),\quad q\in\R
$$
where $
P( q\varphi)=\lim_{n\to+ \infty}\frac{1}{n}\log \int_{[0,1]^d} 2^n \exp\left (S_n (q\varphi)(x)\right)\mathcal{L}^d(\mathrm{d}x)
\in\R $ is the topological pressure of $ q\varphi $.  In particular $\tau_\nu=\lim_{j\to\infty}\tau_{\nu,j}$.
 
\item $\tau_{\nu}$ is analytic, concave increasing and $\lim_{q\to +\infty} \tau_\nu(q) = +\infty$. 
\item

$\nu$ satisfies the multifractal formalism: $\sigma_{\nu}=\tau_{\nu}^*$.
\end{itemize}
Hence, the same properties are true for  any Gibbs capacity $\mu\approx\nu^\gamma$ with $\gamma>0$.

The family of Gibbs capacities is obviously invariant under the operator $\mu \mapsto \mu^\gamma$. It is known that a Gibbs capacity $\mu\approx\nu^\gamma$ is multifractal if and only if the Gibbs measure $\nu$ is not equivalent to the Lebesgue measure, that is when $\varphi$ is not cohomologous to a constant on  $\R^d/\Z^d$ endowed with the dynamics $\times 2$ (see \cite{Bowen}).

\sk

The situation studied in \cite{BS2020}\footnote{In \cite{BS2020} we considered capacities of the form $\mu(I)=K \nu(I)^\alpha 2^{-\beta j}$ for $I\in \mathcal D_j$, with $K>0$, $\alpha\ge 0$, $\beta\ge 0$ and $\alpha+\beta>0$, which is essentially equivalent to assuming $\mu\approx \nu^\gamma$ upon modifying $\varphi$.} corresponds to the extreme  case $\varrho=1/\eta$, that is  $\M_{\ro,\eta}(\mu)= \M(S_\eta(D_1\mu))=\M _\eta(\mu)$. There, we proved that $\M _\eta (\mu)$ still satisfies the multifractal formalism. In the present  paper, we complete the previous results for a Gibbs capacity $\mu$   as follows:
\begin{itemize}
\item when $\varrho\in [1,1/\eta)$, $\M_{\ro,\eta}(\mu)$ obeys the multifractal formalism (although $D_{\varrho\eta}(\mu)$ does not when $\mu$ is multifractal) - see Theorem \ref{main3}  ;
\item
when $0<\varrho<1$ and  $\mu$ is multifractal,    $\M_{\ro,\eta}$  obeys the multifractal formalism only on a strict subset of $\mathrm{dom}(\sigma_{M_{\ro,\eta}})$  - see Theorems \ref{main} and  \ref{main2}.
\end{itemize}
As a consequence of our study, even if all the capacities $D_\varrho (\mu^{1/\varrho})$, for $0<\varrho<1$, share the same singularity spectrum (by the above observations), the capacities $\M_\eta (D_\varrho    (\mu^{1/\varrho})) $,  $0<\varrho<1$, have distinct singularity spectra, and obey the multifractal formalism if and only if $\varrho\in[\eta,1]$. 
 
In this article, we  will focus on the cases $\varrho=1$ and $0<\varrho<1$, since these are the situations where new interesting phenomena occur. We only give a statement for $1<\varrho\le 1/\eta$ since the result and its proof can be  obtained by adapting the proofs of the case $\varrho=1/\eta$ considered in~\cite{BS2020}.

Before stating precisely our results, it is worth saying a word about  the properties of the $L^q$ and singularity spectra of  $\M_{\ro,\eta}(\mu)$ when $\mu$ is Gibbs. While $\tau_\mu$ is analytic,  it will be  shown that for any $\varrho\in (0,1/\eta]$, $\tau_{\M_{\ro,\eta}(\mu)}$ always presents a first order phase transition (i.e.~a point of non differentiability), and in some cases even a second order phase transition (point of differentiability but not twice differentiability). Moreover,  $\tau_{\M_{\ro,\eta}(\mu)}$ is analytic   outside its phase transition(s) point(s). Also, while $\sigma_{\mu}$ is analytic over the interior of its domain  when $\varrho\in (0,1/\eta]$, this is not the case for $\sigma_{\M_{\ro,\eta}(\mu)}$. Indeed, $\sigma_{\M_{\ro,\eta}(\mu)}$ is still  concave but may contain three to five consecutive analytic parts, and can be non differentiable at some point. 
 
\section{Statement of the main results}

A multifractal Gibbs capacity $\mu$ is fixed, as well as  $\eta\in (0,1)$. Set $H_{\min}=\tau_\mu'(+\infty)$ and $H_{\max}=\tau_\mu'(-\infty)$. 
We start with the case $\varrho=1$, which is both the simplest one and a transition between the case $\varrho\in(1,1/\eta]$,  for which the multifractal behavior of $M_{\ro,\eta}(\mu)$ is similar to that in the case $\varrho=1/\eta$ studied in \cite{BS2020} and the multifractal formalism holds, and the case $\varrho\in(0,1)$, which is more versatile and for which the multifractal formalism holds only partially. 
 
\subsection{\bf The case $\varrho=1$}  This situation corresponds to the case where there is an almost compensation between the redundancy process and the random sampling. Indeed, with this choice of parameters, every $\mu(I)$ with $I\in \mathcal{D}_j$ is replaced by  $\mu(I^\eta)$, but only $\sim 2^{jd\eta}$ of them are kept. Since the surviving cubes are chosen  uniformly in $\zu^d$, intuitively the surviving cubes $I$  after sampling are such that the union of the corresponding cubes $I^\eta$ cover $\zu^d$ (this is not true, but not far from it).

\begin{theorem}
\label{main} Let $\mu$ be a multifractal Gibbs capacity, and $\eta\in (0,1)$. Let $  q_1$ be the unique solution of the equation
$\tau_\mu(q)=0$, and let $  H_1= \tau_\mu'(  q_1)$. Equivalently, $   q_1=\sigma_\mu'(  H_1)$.

 With probability 1,  $\M_{1,\eta}(\mu)$ satisfies the multifractal formalism with 
\begin{equation}
\label{spec-M1}
 \sigma_{\M_{1,\eta}(\mu)}(H) = \begin{cases} \      \eta \sigma_\mu(H/\eta) &  \mbox{ if }  \ \eta H_{\min}\le  H < \eta  {H_1} ,\\ 
\  \displaystyle  \frac{ \sigma_{\mu}(H_1) }{H_1} H & \mbox{ if }  \ \eta {H_1} \leq  H< {H_1},  \\   
    \ \ \ \sigma_\mu (H)  & \mbox{ if }   {H_1}\le H<  H_{\max} . \end{cases}
\end{equation}
and $\sigma_{\M_{1,\eta}(\mu)}(H) =-\infty$ otherwise. Also,  
\begin{equation}
\label{tau-M1}
\tau_{\M_{1,\eta}(\mu)}(q)=\sigma_{\M_{1,\eta}(\mu)}^*(q)=
\begin{cases}
\ \tau_\mu(q)&\text{if }q< {q_1},\\
\eta\tau_\mu(q)&\text{if }q\ge   {q_1}.
\end{cases}
\end{equation}
\end{theorem}

The continuity of $ \sigma_{\M_{1,\eta}(\mu)}$ can be checked directly on the formula. The intermediate linear phase is the unique interpolation between $H\mapsto    \eta \sigma_\mu(H/\eta) $ and $H\mapsto \si_\mu(H)$ which makes the resulting spectrum concave.

\subsection{The case $0<\varrho<1$} This situation, in which there is a lot of redundancy in the values $\mu(I^\varrho)$, is quite versatile and involves some extra parameters. 
The existence of the  parameters introduced below, as well as  their justification and their properties, will be studied in Section \ref{pfth2}.

\sk

$\bullet$ 
Let  $ \theta_\varrho$ be the mapping defined as
$$
 \theta_\varrho: H\in[H_{\min},\tau_\mu'(0)]\mapsto \frac{\varrho H{\sigma_\mu(H)}}{ {d(1-\varrho)+\varrho\sigma_\mu(H)}}.
 $$
 It  increases continuously from $\frac{\varrho H_{\min}\sigma_\mu(H_{\min})} {d(1-\varrho)+\varrho\sigma_\mu(H_{\min})} $ to $\varrho \tau_\mu'(0)$.   
 
\sk$\bullet$ Define ${q_\ro}$ as the unique $q\in \R$  such that $\tau_\mu(q)=\frac{d(1-\varrho)}{\varrho}$ and set ${H_\ro}=\tau_\mu'({q_\ro})$.

\mk

Let us now define a mapping, that will be the multifractal spectrum of $\M_{\ro,\eta}(\mu)$. 
  
\begin{definition}\label{defDetarho}
 
\begin{enumerate}
\item Suppose that $\sigma_\mu(H_{\min}) > \frac{d(1-\varrho)}{1/\eta-\varrho}$. Then set
\begin{align}
\label{spec1}
D_{\mu,\ro,\eta}(H)
= \begin{cases}
\eta (d(1-\varrho)+\varrho \sigma_\mu(H/\varrho\eta))&\text{if } \varrho \eta H_{\min}\le H<  \varrho \eta {H_\ro}
,\\
  \ \ \ \ \   \  \ \displaystyle  \frac{\sigma_\mu({H_\ro})}{\theta_\ro ({H_\ro})} H&\text{if }   \varrho \eta {H_\ro} \le H< \thetar ({H_\ro}),\\
 \ \ \ \ \  \sigma_\mu(\thetar^{-1}(H))&\text{if }\thetar ({H_\ro}) \le H< \varrho\tau_\mu'(0),\\
  \ \ \ \  \ \ \ \sigma_\mu(H/\varrho)&\text{if } \varrho \tau_\mu'(0)\le H\le \varrho H_{\max},
  \end{cases}
\end{align}
and $D_{\mu,\ro,\eta}(H)=-\infty$ otherwise. 

\medskip

\item Suppose that  $\sigma_\mu(H_{\min})\le  \frac{d(1-\varrho)}{1/\eta-\varrho}$.   Let $H_{\ro,\eta}  \in [H_{\min},\tau_\mu'(0)]$ be the unique solution to    $$\sigma_\mu(H_{\ro,\eta})=\frac{d(1-\varrho)}{1/\eta-\varrho}.$$
This can be equivalently rewritten as  $ \ro\eta H_{\ro,\eta}=\theta_\ro(H_{\ro,\eta})$
\sk
\begin{enumerate}
\item If $H_{\ro,\eta}<  {H_\ro}$, then  set
\begin{align}
\label{spec2}
D_{\mu,\ro,\eta}(H)
= \begin{cases}
  \ \ \ \ \ \ \ \sigma_\mu(H/\varrho\eta)&\text{if }  \varrho \eta H_{\min}\le  H<  \varrho \eta H_{\ro,\eta},\\
\eta (d(1-\varrho)+\varrho \sigma_\mu(H/\varrho\eta))&\text{if } \varrho \eta H_{\ro,\eta} \le H<  \varrho \eta {H_\ro}
,\\
  \ \ \ \ \ \  \ \  \displaystyle\frac{\sigma_\mu({H_\ro})}{\theta_\ro ({H_\ro})}H&\text{if }   \varrho \eta {H_\ro} \le H< \thetar ({H_\ro}),\\
  \ \ \ \ \ \   \sigma_\mu(\thetar^{-1}(H))&\text{if }\thetar ({H_\ro}) \le H< \varrho\tau_\mu'(0),\\
  \ \ \ \ \ \  \ \  \sigma_\mu(H/\varrho)&\text{if } \varrho \tau_\mu'(0)\le H\le \varrho H_{\max},
  \end{cases}
\end{align}
and $D_{\mu,\ro,\eta}(H)=-\infty$ otherwise.

\begin{center}
\begin{figure}

\includegraphics[scale=0.23]{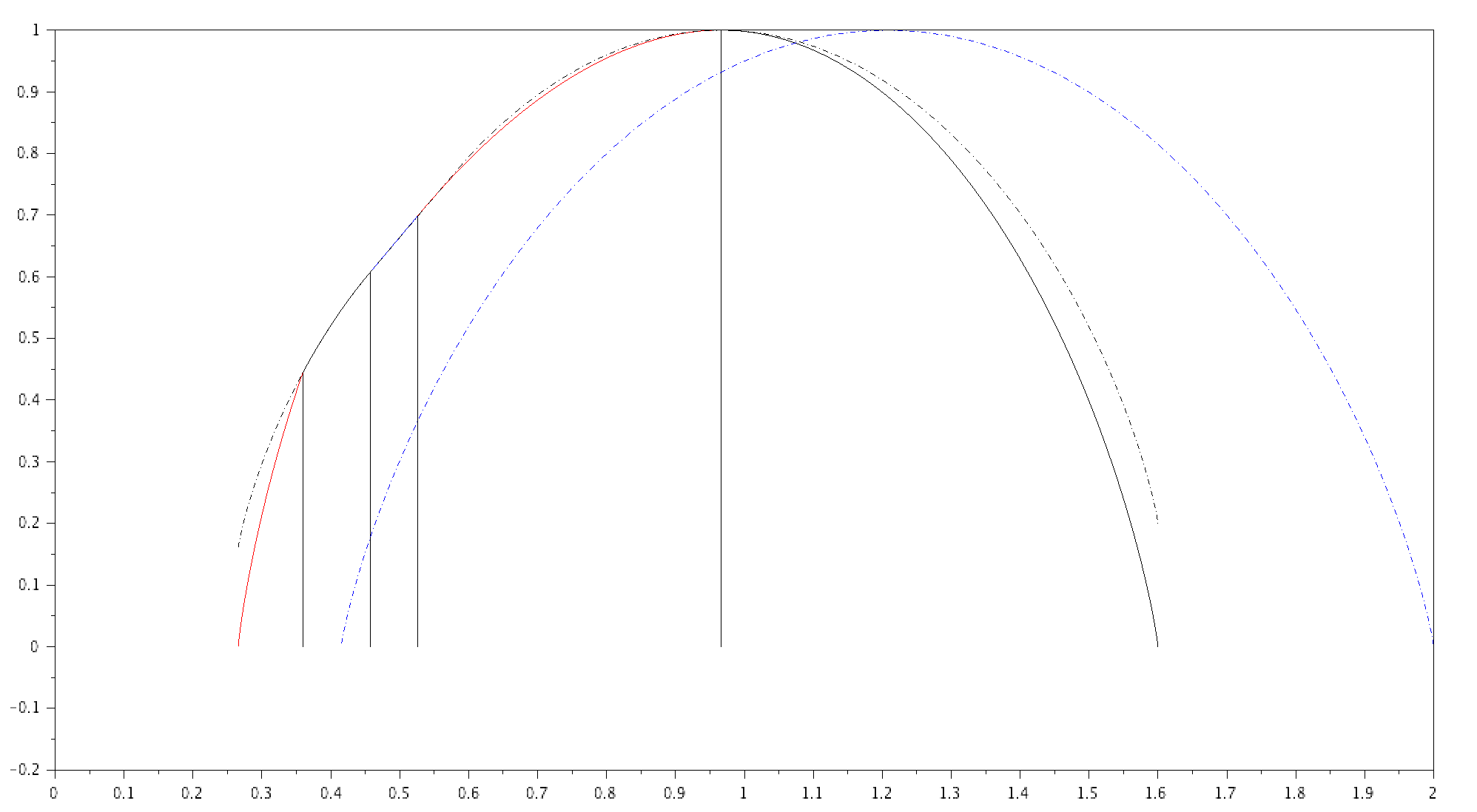}
\includegraphics[scale=0.21]{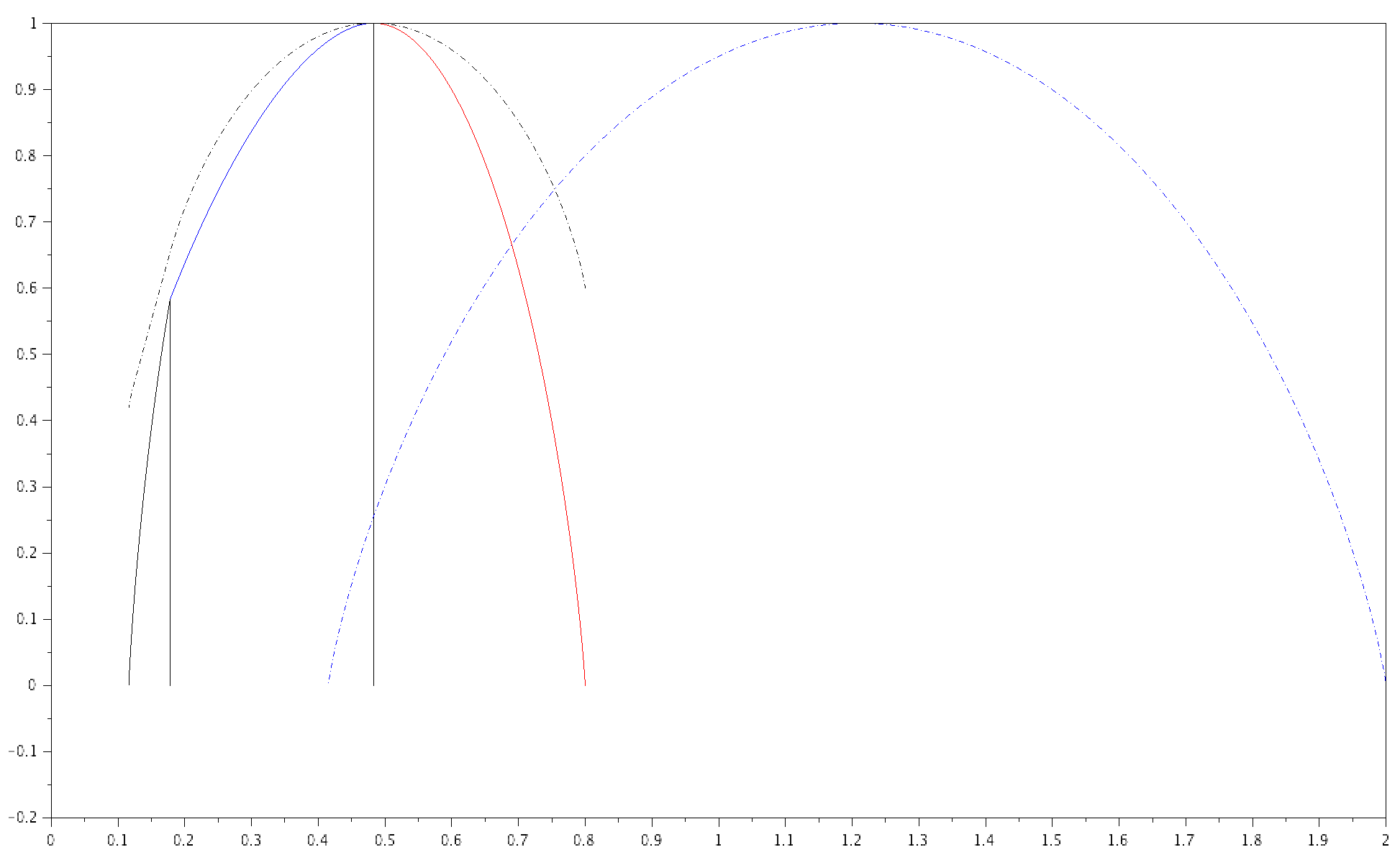}

 \caption{ {\bf Left:}   Spectrum of $\M_{\ro,\eta}(\mu)$ for a given spectrum of $\mu$ when  $\ro=0.8$, $\eta=0.8$, case {\bf 2a}: 5 phases. {\bf Right:}  Spectrum when  $\ro=0.7$, $\eta=0.4$, case {\bf 2b}: 3 phases. Original spectrum in dashed blue, Legendre spectrum in dashed black.}
 \end{figure}
 \end{center}

\item If $H_{\ro,\eta}\ge  {H_\ro}$, then  set
\begin{align}
\label{spec3}
D_{\mu,\ro,\eta}(H)= \begin{cases}
 \  \sigma_\mu(H/\varrho\eta)&\text{if }  \varrho \eta H_{\min}\le H<  \varrho\eta H_{\ro,\eta},\\
 \sigma_\mu(\thetar^{-1}(H))&\text{if } \varrho\eta H_{\ro,\eta}\le H< \varrho\tau_\mu'(0), \\
\ \sigma_\mu(H/\varrho)&\text{if } \varrho \tau_\mu'(0)\le H\le \varrho H_{\max} , \end{cases}
\end{align}
and $D_{\mu,\ro,\eta}(H)=-\infty$ otherwise. 

\end{enumerate}
\end{enumerate}
\end{definition}

Observe that when $H_{\ro,\eta}=H_\ro$, then $\ro\eta H_{\ro,\eta} = \ro\eta H_\ro=\theta_\ro(H_\ro)$, and the second and third intervals in \eqref{spec2} merge to a single point. These intervals do not matter in \eqref{spec3} when  $H_{\ro,\eta}\ge  {H_\ro}$.

The formulas are not   easy to handle with, but they are driven by the model. In particular, in each case there are different phases. The continuity of  $D_{\mu,\ro,\eta}$ can be checked  on the formulas, and its other properties  (differentiability, concavity) are studied  in Section \ref{pfth2} and  gathered in the next proposition. One can already notice that in \eqref{spec1} and \eqref{spec2} there is a linear part in the spectrum.

\begin{proposition}\label{regularityDetarho}For all $\varrho\in(0,1)$, the function $D_{\mu,\ro,\eta}$ is concave, continuous, and  analytic on every open   interval    given by the different formulas in Definition~\ref{defDetarho}. Moreover, $D_{\mu,\ro,\eta}$ is differentiable, except at $\varrho\eta H_{\ro,\eta}$ in cases (2a) and (2b) of Definition~\ref{defDetarho}. 
\end{proposition}

Our main theorem is the following.

\begin{theorem}\label{main2} Let $\mu$ be a multifractal Gibbs capacity on $\zu^d$, and $\eta\in (0,1)$.   For all $\varrho\in (0,1)$, with probability 1:
\begin{enumerate}
\item $\sigma_{\M_{\ro,\eta}(\mu)}=D_{\mu,\ro,\eta}$. 
\item One has  
$$
\tau_{\M_{\ro,\eta}(\mu)}(q)=
\begin{cases}
d(\varrho-1)+\varrho\tau_\mu(q)&\text{if }q<{q_\ro},\\
\eta(d(\varrho-1)+\varrho\tau_\mu(q))&\text{if }q\ge {q_\ro}.
\end{cases}
$$
\item
Set 
$$I_{\ro,\eta}= \begin{cases}  [\varrho \eta H_{\min}, \theta_\varrho(H_\varrho)]\cup\{\varrho\tau_\mu'(0)\} & \mbox{ if } \sigma_\mu(H_{\min})>  \frac{d(1-\varrho)}{1/\eta-\varrho} \\
[\varrho \eta{H_{\ro,\eta}}, \theta_\varrho(H_\varrho)]\cup\{\varrho\tau_\mu'(0)\} &\mbox{ if }\sigma_\mu(H_{\min})\le  \frac{d(1-\varrho)}{1/\eta-\varrho} \mbox{ and }H_{\varrho,\eta}<{H_\ro},\\
 \{\varrho \eta{H_{\ro,\eta}}\}\cup\{\varrho\tau_\mu'(0)\} & \mbox{ if } \sigma_\mu(H_{\min})\le  \frac{d(1-\varrho)}{1/\eta-\varrho}  \mbox{ and }H_{\varrho,\eta}={H_\ro}\\
\{\varrho\tau_\mu'(0)\} & \mbox{ if } \sigma_\mu(H_{\min})\le  \frac{d(1-\varrho)}{1/\eta-\varrho}  \mbox{ and }H_{\varrho,\eta}>{H_\ro} .\end{cases}
 $$
The multifractal formalism holds for $\M_{\ro,\eta}(\mu)$ over $I_{\ro,\eta}$ and it fails over $\mathrm{dom}(\sigma_{\M_{\ro,\eta}(\mu)})\setminus I_{\varrho,\eta}$. 
\end{enumerate}
\end{theorem}

So, we are able to compute in each case the multifractal spectrum and the $L^q$-spectrum, but these two quantities do not form a Legendre pair (i.e. they are not Legendre transform one  of each other). A very careful analysis must be done to find a sharper upper bound for $\si_{\M_{\ro,\eta}(\mu)}$ than the one provided by the multifractal formalism. 

One may wonder if alternatively the multifractal formalism developed by Olsen in~\cite{Olsen} does hold for $\M_{\ro,\eta}(\mu)$. However this formalism is tailored to study the Hausdorff dimension of the level sets of the local dimension, while (for all $\rho$), these sets $E_{\M_{\ro,\eta}(\mu)}(H)$ (see Definition~\ref{def3.1}) are empty for $H$ in a non-trivial subinterval of $\mathrm{dom}(\sigma_{\M_{\ro,\eta}(\mu)})$.

\subsection{The case $1< \varrho \le 1/\eta$} 

In this case, there is less redundancy.
For $\varrho\in (1,1/\eta]$, let us introduce some parameters:
\begin{itemize}
\sk\item 
Let $q_\varrho$ be the unique solution of $\varrho-1+\varrho\tau_\mu(q)=0$, and $H_\varrho=\tau_\mu'(q_\varrho)$.

\sk \item
 Let  $\widetilde H_\varrho =\min\{H\geq 0: \sigma_\mu(H) \geq d(1-1/\varrho)\}$.   
  \sk\item
Finally,   $  \widehat H_\varrho  =- \frac{\tau_\mu(q_\varrho) }{ q_\varrho} $.
\end{itemize} 
 
\begin{theorem}\label{main3} 
Let $\mu$ be a multifractal Gibbs capacity, and $\eta\in (0,1)$. For all $\varrho\in (1,1/\eta]$, with probability 1,  $\M_{\ro,\eta}$  satisfies the multifractal formalism with
$$
\sigma_{\M_{\varrho,\eta}(\mu)}(H)=
\begin{cases}
\eta(1-\varrho +\varrho\sigma_\mu(H/\varrho\eta)) &\text{if }  \ \ \sk \varrho\eta \widetilde H_\varrho  \le H < \varrho\eta  {H_\varrho} ,\\ 
 \ \ \  \ \  \ \  \ \  q_\varrho  H&\text{if }  \displaystyle  \ \ \varrho\eta {H_\varrho} \le  H <{H_\varrho} +  \widehat H_\varrho ,\\
\ \ \ \ \ \  \  \sigma_\mu \big (H-  \widehat H_\varrho \big)&\text{if }  \ \   {H_\varrho}+ \widehat H_\varrho \le  H \leq   \tau_\mu'(\infty)+  \widehat H_\varrho
\end{cases}
$$
and $\sigma_{\M_{\ro,\eta}(\mu)}(H) =-\infty$ otherwise. In particular, 

$$
 \tau_{\M_{\ro,\eta}(\mu)}(q)=\sigma_{\M_{\ro,\eta}(\mu)}^*(q)=
\begin{cases}
\ \tau_\mu(q)- \frac{\tau_\mu(q_\varrho)}{ q_\varrho} \cdot  q&\text{if } \ \ \sk\sk  q\le  q_\varrho,\\\sk\sk
\ \eta(\varrho-1+\varrho\tau_\mu(q))&\text{if } \ \  q_\varrho <q< \widetilde q_\varrho,  \\ \sk\sk
 \ \eta\varrho \tau_\mu'(\widetilde q_\varrho) \cdot  q&\text{if  \ \  $\widetilde q_\varrho<+\infty$ and  $q\ge \widetilde q_\varrho$}.
\end{cases}
$$
   \end{theorem}

There are two phase transitions in   $\sigma_{\M_{\ro,\eta}(\mu)}$, and one or two in $ \tau_{\M_{\ro,\eta}(\mu)}$.

\begin{center}
\begin{figure}
\includegraphics[scale=0.21]{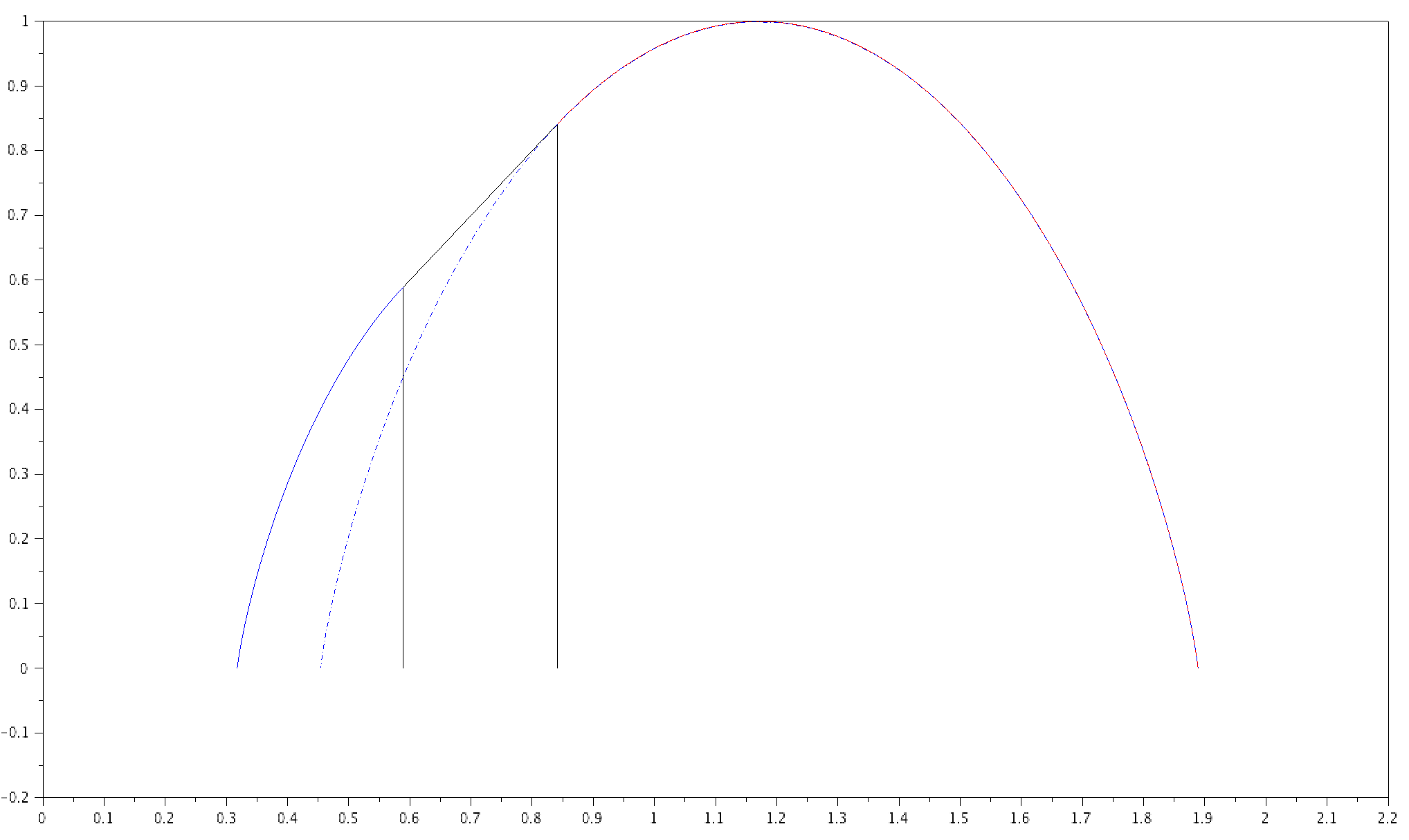}
\includegraphics[scale=0.22]{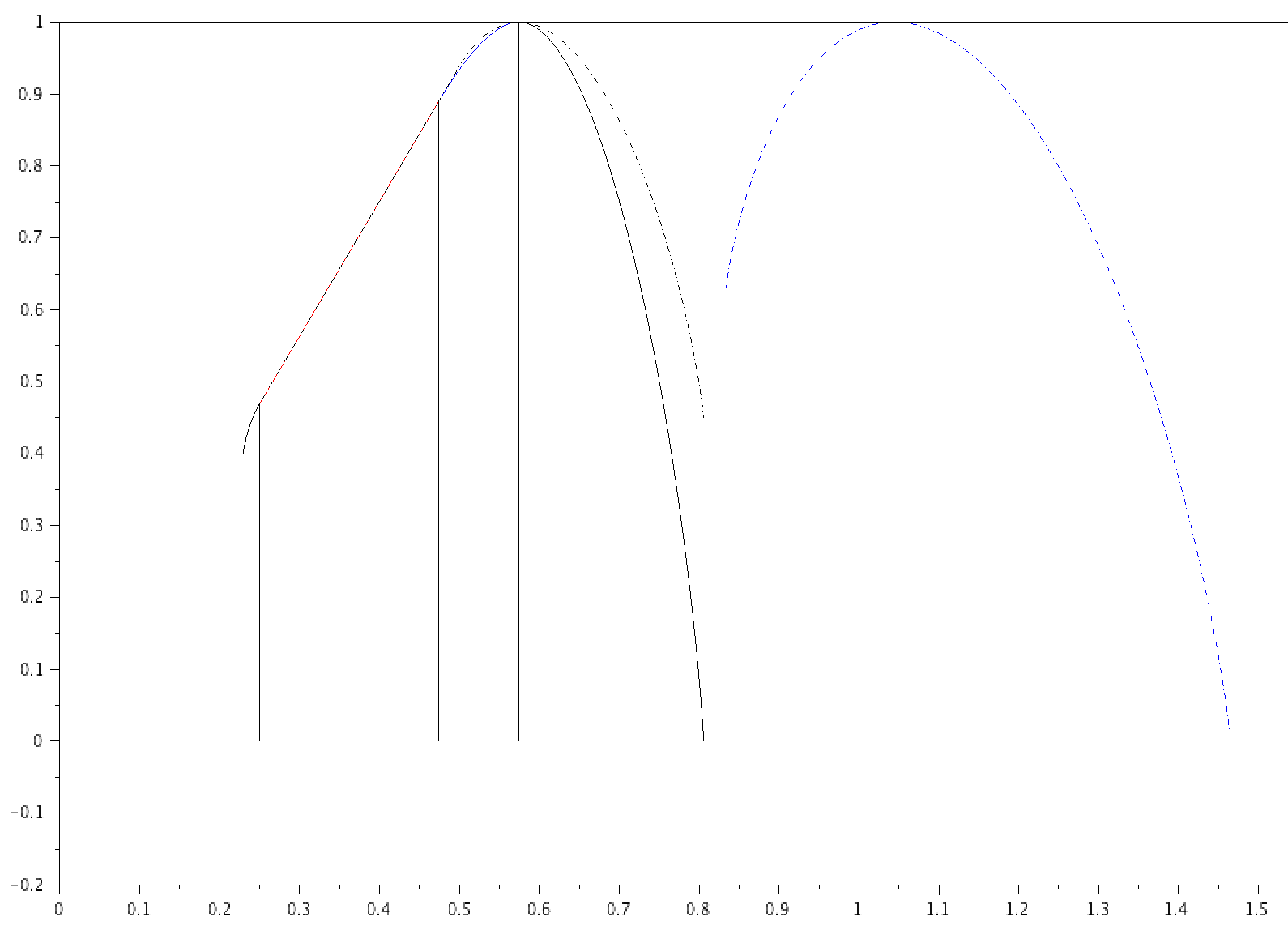}
  \caption{{\bf Left:} Spectrum of $\M_{1,\eta}(\mu)$ for a given spectrum of $\mu$ when $\eta=0.7$: 3 phases. {\bf Right:} Spectrum of $\M_{\ro,\eta}(\mu)$ when $\ro=0.55$, $\eta=0.5$,  case 1: 4 phases. Original spectrum in dashed blue, Legendre spectrum in dashed black (for $\rho<1$ only).}
 \end{figure}
 \end{center}

This situation for $\varrho\in(1,1/\eta)$ is very similar to the the case $\varrho=1/\eta$ studied in \cite{BS2020}, and with some work, one can adapt easily the arguments of \cite{BS2020}. Note that in \cite{BS2020}, when $\varrho=1/\eta$, we define $H_\varrho$ equivalently as  the unique real number such that the tangent to the graph of  $\sigma_\mu$ at the point $( {H_\varrho},\sigma_\mu( {H_\varrho}))$ passes through $(0,d(1-1/\varrho))$, and $\widehat H_\varrho$ as $- \frac{\tau_\mu(\sigma_\mu'(H_\varrho) )}{ \sigma_\mu'(H_\varrho)}$. These equivalent  formulations do hold when $\varrho\in (1,1/\eta)$. 

We choose not to give the proof of Theorem \ref{main3}, and to focus on Theorems \ref{main} and \ref{main2}, since  these are two situations where new phenomena appear, in particular where the multifractal formalism does not hold any more.

Let also note that the graph of the multifractal spectrum of $\M_{\varrho,\eta}(\mu)$  depends continuously (in the Hausdorff distance metric) on $\eta\in (0,1)$ and  $\varrho\in (0,1/\eta]$.

\subsection{Lacunary wavelet series}\label{LWS}

Let us quickly explain the connexion with lacunary wavelet series. In \cite{JAFF_lac}, Jaffard studied real  random processes  called lacunary wavelet series, essentially equivalent  to the following model when $d=1$: for $\gamma>0$ and $\eta<1$,let 
$$L_{\gamma,\eta}(x) = \sum_{j\geq 1} \sum_{k=0}^{2^j-1} 2^{-j\gamma}X_{j,k} \psi(2^jx-k), \ \ \ x\in\R,$$
where $\psi$ is a wavelet (see \cite{Meyer_operateur} for precise definitions), and the $X_{j,k} $ are independent random variables such that each $X_{j,k}$ is a Bernoulli variable with parameter $2^{-j\eta}$. Recalling Definition \ref{defM}, it turns out \cite{JaMe,JAFFARDPSPUM} that analyzing the pointwise regularity and the multifractal properties  of $\mu$ is equivalent to performing  the multifractal analysis the capacity $\M( c_w)$, where the sequence $c_w$ is the wavelet coefficients of $L_{\gamma,\eta}$: $c_w=(2^{-j\gamma}X_{j,k})_{j\geq 1, k\in \{0,...,2^j-1\}}$.

Hence our model can be seen as a generalisation of the $L_{\gamma,\eta}$ process, where  for $I\in \mathcal{D}_j$ the coefficient $2^{-j\gamma} $ (which can be interpreted as $(\mathcal{L}^1(I))^\gamma)$  or $\mathcal{L}^1(I^\gamma)$) is replaced by $\nu(I)^\gamma$ for a Gibbs measure $\mu$. In other words, instead of working from the uniform Lebesgue measure, we start  with a multifractal  measure  and obtain a random process with a much richer structure.

Let us mention \cite{Esser-Vedel} where some variations of Jaffard’s model are considered to construct random functions supported on Cantor sets, and whose restrictions to their support have a piecewise linear non-decreasing multifractal spectrum and do not obey the multifractal formalism. Other interesting multifractal random wavelets series are considered in \cite{AubryJaffard,BSFmu,durand2008treeindexedchain}.

\subsection{Final remark on $\tau_{\Mr}$ and the large deviations spectra of  $\M_\ro$}\label{LD}

\begin{definition} 
\label{defsmu}
Let $\mu\in  \mathrm{Cap}([0,1]^d) $ with $\supp(\mu)\neq\emptyset$.
For $H\in\R$, $j\in \N$ and $\ep>0$, let
$$\mathcal{E}_\mu(j,H\pm \ep) =
  \left\{ w\in \Sigma_j: \frac{\log_2 \mu(I_w)}{ {-j}} \in
[H-\ep,H+\ep] \right\}.$$
Then, the lower and upper large deviations spectra of $\mu$ are respectively defined as
\begin{eqnarray*} 
\underline {f}_\mu(H)&= &  \lim_{\ep\to 0}\liminf_{j\to+\infty} \frac{\log_2 \#\mathcal{E}_\mu(j,H\pm\ep) }{j}  \\
\mbox{ and } \ \  \   \overline {f}_\mu(H)  & =  & \lim_{\ep\to 0}\limsup_{j\to + \infty} \frac{\log_2 \#\mathcal{E}_\mu(j,H\pm\ep)}{j}.
\end{eqnarray*}
 \end{definition}

 Let now $\mu$ be a multifractal Gibbs capacity. The equality $\sigma_\mu(H)= \tau_\mu^*(H)$ can be strengthened into $\dim E_\mu(H)=\tau_\mu^*(H)$ for all $H\in\R$ \cite{ColletLebPor,Pesin} (see Definition~\ref{def3.1} for the definition of $E_\mu(H)$); together with the standard inequalities $\dim E_\mu(H)\le \underline{f}_\mu(H)\le \overline{f}_\mu(H)\le \tau_\mu^*(H)$ (valid for any capacity \cite{Olsen,LN,JLVVOJAK}), this yields  $\underline f_\mu=\overline f_\mu=\tau_\mu^*$ \cite{Rand,Olsen,FanFeng}. Alternatively, this comes from the facts that $\tau_\mu$ is differentiable and $\tau_\mu=\lim_{j\to\infty}\tau_{\mu,j}$, and the Gartner-Ellis theorem~\cite{DemboZeitouni}. It is not hard to deduce from the studies conducted in this paper and its companion~\cite{BS2020} that for all $\varrho\in (0,1/\eta]$, with probability 1, one has $\tau_{\M_{\varrho,\eta}(\mu)}=\lim_{j\to\infty}\tau_{\M_{\varrho,\eta}(\mu),j}$ and $\underline f_{\M_{\varrho,\eta}(\mu)}=\overline f_{\M_{\varrho,\eta}(\mu) }=\tau_{\M_{\varrho,\eta}(\mu)}^*$.
\subsection{Organisation of the article}

In Section 3 we recall additional properties associated with the multifractal formalism of Gibbs capacities and measures, as well as some properties of the random sampling process that were already established in \cite{BS2020}. In Section 4, the case $\rho=1$ is treated (Theorem \ref{main}). Section 5 contains the proof of the redundant case $\rho\in (0,1)$, Theorem \ref{main2}. 
\section{More on the multifractal formalism, Gibbs capacities, and the random sampling}

\subsection{Some notations} For every $j\geq 1$, let $\Sigma_j=(\{0,1\}^d)^j$. Let  $\Sigma^*=\bigcup_{j\ge 0} \Sigma_j$ and $ \Sigma=(\{0,1\}^d)^\N$  be respectively the sets of finite and infinite words over the alphabet $\{0,1\}^d$. By convention $(\{0,1\}^d)^0$ contains the empty word denoted by $\epsilon$. The set $\Sigma^*$ acts on $\Sigma^* \cup\Sigma$ by concatenation, that is if $u\in \Sigma^*$ and $v\in \Sigma^* \cup\Sigma$, the concatenation $u\cdot v$ of $u$ and $v$ is obtained by adding the prefix $u$ to the word $v$. If $u\in\Sigma^*$, it defines the cylinder $[u]$ consisting of all infinite words with common prefix $u$.  Let $\mathcal C=\{[u]:u\in \Sigma^*\}$.  The length of an element $u$ of $\Sigma^* \cup\Sigma$, that is its number of letters, will be denoted by $|u|$. If $0\le j\le |u|$, $u_{|j}$ stands for the prefix of $u$ of length $j$. The set $\Sigma$ is endowed with the standard ultrametric distance $\mathrm{d}:(x,y)\mapsto 2^{-|x\land y|}$, where $x\land y$ is the longest common prefix of $x$ and $y$. 

The coding map 
$
\pi: t=(t_n^{(1)},\ldots,t_n^{(d)})_{n\ge 1}\in \Sigma\mapsto \left (\sum_{n=1}^\infty t^{(i)}_n2^{-n}\right)_{1\le i\le d}\in[0,1]^d
$
maps each cylinder $[w]\in \mathcal C$ onto the closed nontrivial dyadic cube  
$$
I_w=\prod_{i=1}^d\left[\sum_{k=1}^{|w|}w^{(i)}_k2^{-k}, 2^{-|w|}+\sum_{k=1}^{|w|}w_n^{(i)}2^{-k}\right].
$$
We define 
\begin{equation}
\label{defxw}
x_w= \left(\sum_{k=1}^j w^{(i)}_k 2^{-k} \right)_{1\le i\le d}.
\end{equation}
 If $x=(x^{(1)}, x^{(2)},\ldots,x^{(d)})\in [0,1]^d$ has no dyadic component, then $x$ is encoded by a unique $w=(w^{(1)},w^{(2)},\ldots,w^{(d)}) \in \Sigma$, and $I_j(x)$ stands for $I_{w_{|j}}$. When $x^{(i)}$ is dyadic,   $w^{(i)}$ is chosen as the largest element of $\{0,1\}^{\N_+}$ in lexicographical order which encodes $x^{(i)}$. In both cases, $w_{|j}$  is also denoted $x_{|j}$.

\subsection{Additional properties associated with multifractal formalism and Gibbs capacities}
\begin{definition}\label{def3.1}
Let  $\mu  \in \mathrm{Cap}([0,1]^d) $ with full support.  
For $x\in [0,1]^d$, the lower and upper  local dimensions of $\mu$ at $x$ are respectively defined as 
$$
{\underline \dim_\locloc}(\mu,x)=\liminf_{r \to 0^+}\frac{\log \mu(B(x,r))}{ \log r}  \ \text{ and } \  {\overline \dim_\locloc}(\mu,x)=\limsup_{r\to 0^+}\frac{\log \mu(B(x,r))}{ \log r }.
$$

When ${\underline \dim}(\mu,x) = {\overline \dim_\locloc}(\mu,x)$,  their common value is denoted by $ {\dim_\locloc}(\mu,x)$.

For $H\in\R$, set
\begin{equation*}
\begin{split}
\underline E_\mu(H)&=\left \{x\in[0,1]^d: {\underline \dim_\locloc}(\mu,x)=H\right \},\\
 \overline E_\mu(H)&=\left \{x\in[0,1]^d: {\overline \dim_\locloc}(\mu,x)=H\right \},\\
  E_\mu(h)&= \underline E_\mu(H)\cap   \overline E_\mu(H).
  \end{split}
 \end{equation*}
\end{definition}

Observe that for every $x\in \zu^d$, 
\begin{equation}
\label{equiv-dim}
{\underline \dim_\locloc}(\mu,x)=\liminf_{j \to +\infty}\frac{\log \mu\left (\bigcup_{I\in \mathcal{N}(I_j(x))}I \right) }{ -j}=\liminf_{j \to +\infty}\frac{\log  \max \Big(\mu(I) :  I \in \mathcal{N}(I_j(x) ) \Big) }{ -j},
\end{equation}
where  $I_j(x)$ stands for the unique dyadic cube of generation $j$ that contains $x$, and where the notation $\mathcal{N}(I)$  for $I\in \mathcal{D}_j$ stands for the family of dyadic cubes neighbor to $I$ of same generation:
\begin{equation}
\label{defNj}
\mbox{for } I\in \mathcal{D}_j, \ \  \mathcal{N}(I) =\{  I' \in \mathcal{D}_j: \ \overline{I} \cap \overline{I'}  \neq \emptyset \}
 \end{equation}
contains $I_j(x)$ and its $3^d-1$ neighboring dyadic cubes at generation $j$. By abuse of notations, we will sometimes write   $w\in \mathcal{N}(I)$ for $I\in \mathcal{D}_j$ to describe all the dyadic cubes $I_w \in \mathcal{D}_j$ belonging to~$\mathcal{N}(I)$.
  
\begin{definition}\label{defnewsets}
For any fully supported capacity $\mu\in \mathrm{Cap}([0,1]^d)$, define  the level sets 
\begin{eqnarray*}
\underline{E}^{\leq}_\mu(H) &=\{x \in [0,1]^d:\, {\dimi_\locloc}(\mu,x) \leq H\}  \\
\text{and }  \ \  \overline{E}^{\geq}_\mu(H) &=\{x \in [0,1]^d:\, {\dims_\locloc}(\mu,x) \geq H \}.
\end{eqnarray*}
\end{definition}
The multifractal formalism also claims  that 
\begin{align}
\label{upcap1}
&\dim \underline{E}^{\leq}_\mu(H)\le \tau_\mu^*(H) \text{ if } H\le \tau_\mu'(0^-),\\
\label{upcap2}
\text{and } &\dim \overline{E}^{\geq}_\mu(H)\le \tau_\mu^*(H) \text{ if } H\ge \tau_\mu'(0^+).
\end{align}

Let now $\mu$ be  a multifractal Gibbs capacity $\mu$. One always has:
{\begin{itemize}
\sk\item
$\tau_\mu$ is a limit (not only a liminf), i.e. 
\begin{equation}
\label{lim-tau} 
\tau_\mu(q) =  \lim_{j\to\infty} \tau_{\mu,j}(q), \ \ \ \mbox{ for every }q\in \R. 
\end{equation}
 \item $\tau_\mu$ is strictly concave, analytic (due to the properties recalled in the previous section), and $\sigma_\mu$ is strictly concave,  and real analytic   over $(H_{\min}, H_{\max})$. 
\item $\mu$ is quasi-multiplicative : There exists a constant $C\geq 1$ such that for every finite words $w,w'\in \Sigma^*$,
\begin{equation}
\label{quasib}
C^{-1} \mu(I_w) \mu(I_{w'}) \leq \mu(I_{ww'})\leq C \mu(I_w) \mu(I_{w'}) .
\end{equation}

\item
There are two exponents $0< \alpha_m\leq \alpha_M <+\infty$   such that for every finite word  $w \in \Sigma^*$,
\begin{equation}
\label{uniform}
 2^{-|w|\alpha_M}\leq \mu(I_{w})\leq 2^{-|w|\alpha_m} .
\end{equation}

\item $\mu$ is doubling.
\end{itemize}}
These properties will be used repeatedly in the following.

\subsection{Basic properties of the distribution of the surviving vertices}

All these properties are proved in \cite{BS2020}. If $w\in\Sigma^*$, then $p_{I_w}$ is simply denoted $p_w$.

 \begin{definition}\label{def11}
If $w\in\Sigma^*$ and $p_w=1$, say that  $w$ is a surviving vertex (or a survivor).

  For every $j\geq 1$, denote by $\mathcal{S}_j(\eta)$   the (random) set of  surviving vertices   in $ \Sigma_j$:
    $$ \mathcal{S}_j(\eta):= \big\{ w\in \Sigma_j:  p_w =1\}.$$
\end{definition}

Recall that  $x_w$, defined by \eqref{defxw}, is the dyadic point corresponding to the projection of the finite word $w\in \Sigma_j$ to $\zu^d$. The first question concerns the distribution of the  points $x_w$,  for $w\in \sjeta$.

 \begin{definition}
 For every $j\geq 1$, and every finite word $W\in \Sigma^*$, one sets
 $$  \mathcal{S}_j(\eta,W) = \{w\in  \mathcal{S}_j(\eta): I_w \subset I_W\}.$$    
 \end{definition}

The set $ \mathcal{S}_j(\eta,W) $ describes the  surviving vertices at generation $j$ included in $I_W$.
 Obviously, for every $J\leq j$, 
 $$ \mathcal{S}_j(\eta) =\bigcup _{W\in \Sigma_J}   \mathcal{S}_j(\eta,W) .$$

 \begin{lemma}\cite{BS2020}
 \label{lem1}
There exists a positive sequence $(\ep_j)_{j\geq 1}$ converging to 0 such that, with probability 1,    for  every   $j$ large enough,  for every $W\in \Sigma_{\lfloor j(\eta-\ep_j)\rfloor}$, one has   $\mathcal{S}_j(\eta,W) \neq \emptyset$.  
 \end{lemma}
 
 So, every cylinder of generation $\lf j(\eta-\ep_j)\rf$ contains a surviving vertex $w$ of generation $j$.
 
{\bf The sequence $(\ep_j)_{j\geq 1}$ is now fixed.}

 Lemma \ref{lem1} has the following consequence: Almost surely,  the set of points belonging to an infinite number of balls of the form $B(x_w, 2^{-\lfloor |w|(\eta-\ep_{|w|})\rfloor})$ with $p_w=1$  is exactly  the whole cube $[0,1]^d$, i.e.
 \begin{equation}
 \label{cover1}
 [0,1]^d = \limsup_{j\to +\infty} \ \bigcup_{w\in \mathcal{S}_j(\eta)} B(x_w, 2^{-\lfloor |w|(\eta-\ep_{|w|})\rfloor}).
 \end{equation}
  
 Next we obtain an upper bound for the cardinality of $\mathcal{S}_j(\eta,W)$ when $W \in \Sigma_{\lfloor \eta j \rfloor}$.
 
 \begin{lemma}
 \label{lem2}
  With probability one,  for  every large  $j$ and  every $W \in \Sigma_{\lfloor \eta j \rfloor}$,   $ \#  \mathcal{S}_j(\eta,W) \leq j$.
 \end{lemma}

As a conclusion, one keeps in mind the intuition that every cylinder $W\in \Sigma_{\lf \eta j  \rf}$ contains at least one, but not much more than one surviving vertex $w\in \mathcal{S}_j(\eta)$.

 \section{Proof of Theorem~\ref{main}}

Recalling Definitions \ref{defMetarho}, one has
 $$
\M_{1,\eta}(\mu)=\M_\eta (D_{\eta} \mu)=\M(S_\eta (D_{\eta} \mu)).
$$
 
 By Definitions \ref{defM} and \ref{defMeta}, each $\M_{1,\eta}(\mu) (I)$, for $I\in \mathcal{D}_j$, can be written as
 $$\M_{1,\eta}(\mu) (I) = \sup_{I'\subset I, \ p_{I'}=1} \mu((I')^\eta).$$
 For sake of simplicity,  $\M_{1,\eta}(\mu)$ is denoted by $\M_1$ in this section.

 \subsection{{\bf Upper and lower bound for the weights of $\M_1$ on dyadic cubes}}

By some abuse of notations, we sometimes identify a dyadic cube $I\subset \zu^d$ of generation $j$ with the corresponding word $w\in \Sigma_j$.
 
\begin{lemma}
\label{lem3}
With probability 1,  there exists a positive sequence $(\widetilde \ep_j)_{j\geq 1}$ such that for every $w\in \Sigma^*$,  
\begin{equation}
\label{majmin0}
 \mu(I_w) 2^{-|w|\widetilde \ep_{|w|} } \leq    \M_1(I_w)    \leq  \     \mu(I^\eta_w).
\end{equation}
\end{lemma}

\begin{proof}
Fix $w \in \Sigma_j$. In order to compute $\M_1(I_w)$, one needs to find the largest value $\mu((I')^\eta)$  for $I'\subset I_w$ and  $p_{I'}=1$. 
 
\sk

The upper bound in \eqref{majmin0} simply follows from the hierarchical structure of the initial Gibbs capacity $\mu$. Indeed, the greatest value  that may appear in the computation of $\M_1(I_w)$ is $\mu(I_w^\eta)$. 

\mk

Further, we estimate the lower bound. The covering lemma \ref{lem1} ensures that when $j=|w|$ is large enough, there is at least one surviving vertex $w'\in \mathcal{S}_{\lfloor j/(\eta-\ep_j)\rfloor }(\eta)$ such that $I_{w'}  \subset I_w$. Hence 
$$\M_1(I_w) \geq \mu(I_{w'}^\eta) .$$
The cube $I_{w'}^\eta$ corresponds to a cube $I_{\widetilde w}$, where $\widetilde w$ is the prefix of the word $w'$ of length $\big \lfloor \eta |w'| \big \rfloor= \big \lfloor\eta \lfloor j/(\eta-\ep_j)\rfloor     \big \rfloor $, which satisfies for large $j$ (i.e. for $\ep_j$ small) 
$$j(1+\ep_j/\eta)-\eta-1  \leq \big \lfloor \eta |w'| \big \rfloor  \leq j(1+2\ep_j/\eta).$$
Since $I_{w'} \subset I_w$,    $w$ is a prefix of the word corresponding to the cube  $I_{w'}^\eta$, so we can write using the quasi-Bernoulli property \eqref{quasib} of $\mu$ that 
$$          
 \frac 1 C \mu(I_w) \mu(J) \leq \mu(I_{w'}^\eta ) \leq C \mu(I_w) \mu(J) 
 $$
for some dyadic cube $J$ of generation between $j\ep_j/\eta-\eta-1$    and $ j 2\ep_j/\eta$.

Using the uniform bounds  \eqref{uniform} for $\mu$ to bound $\mu(J)$, recalling $j=|w|$, we  have
$$      \frac 1 C \mu(I_w)  2^{-   (|w| \ep_{|w|} \/\eta-\eta-1) \alpha_M }  \leq \mu(I_{w'}^\eta ) \leq C \mu(I_w)  2^{-  (|w| \ep_{|w|} \/\eta-\eta-1)\alpha_m}. $$
  This can be rewritten, for some positive sequence $(\widetilde \ep_j)_{j\geq 1}$ converging to zero, as
 \begin{equation}
 \label{voila}
   2^{-|w| \widetilde\ep_{|w|}} \leq  \frac{\mu(I_{w'}^\eta)}{\mu(I_w) }  \leq 2^{|w| \widetilde\ep_{|w|} },
 \end{equation}
 hence the result.
 \end{proof}

 \subsection{Local dimensions of $\M_1$, $\mu$, and approximation rates}

The value of the local dimensions of $\M_1$ at a point $x\in \zu^d$ depends on how close to the "surviving" wavelet coefficients $x$ are. Hence, we introduce  an approximation rate and some   following approximating sets.

 \subsubsection{{\bf Approximation rate and local dimensions of $\M_1$}}

Let us  denote, for every $w\in \Sigma^*$, 
 \begin{equation}
 \label{deflw}
 \ell_w=2^{-\lfloor |w|(\eta-\ep_{|w|})\rfloor }.
 \end{equation}

 Recall that by \eqref{cover1},  almost surely, one has the covering property 
$$
 \zu^d = \limsup_{j\to +\infty} \ \bigcup_{k\in \mathcal{S}_j(\eta)} B(x_w, \ell_w).$$

\begin{definition}
\label{defAdelta}
For every $\delta\geq 1$, we define
$$ \mathcal{A}_\delta = \limsup_{j\to +\infty} \ \bigcup_{k\in \mathcal{S}_j(\eta)} B \big(x_w, (\ell_w)^ \delta\big).$$
\end{definition}

The set $\mathcal{A}_\delta$ contains those points $x$ which are approximated at a rate $\delta\geq 1$ by those dyadic numbers $x_w$ such that $p_w=1$.  Obviously, $\mathcal{A}_1=[0,1]^d$ almost surely.
 
\begin{definition}
\label{def-approx}
For every $x\in \zu^d$, the approximation rate of $x$ is the real number
$$\delta_x=\sup\{\delta\geq 1: x\in   \mathcal{A}_\delta \}.$$
\end{definition}

A remarkable property is that, at the points $x$ which are badly approximated by the surviving dyadic numbers, the local dimension of $\M_1$ is easily expressed in terms of the local dimension of $\mu$.     

\begin{proposition}
\label{prop2}
Almost surely, for every $x\in \zu^d$, if $\delta_x=1$, then  $\dimi(\M_1,x)=\dimi (\mu,x)$.
\end{proposition}
\begin{proof}

Let $x \in \zu^d$  be such that $\delta_x=1$. Recall that $I_j(x) \in \mathcal{D}_j$ is the unique dyadic cube of generation $j$ that contains $x$.
Our goal is to estimate $\M_1(B(x,2^{-j}))$. Observe that by \eqref{equiv-dim}, it is equivalent to estimate 
\begin{equation}
\label{defleaderM}
\wM_1(I_j(x))= \max\Big(\M_1(I): I\in \mathcal{N}(I_j(x))) \Big),
\end{equation}
 where  $\mathcal{N}(I)$ was defined in \eqref{defNj}.  

Since $\delta_x=1$, for every $\ep>0$, $x$ does no belong to $\mathcal{A}_{1+\ep}$. In other words, there exists an integer $J_{\ep,x}$ such that for any $j\geq J_{\ep,x}$,  $x$ does not belong to $\bigcup_{w\in \mathcal {S}_j(\eta)} B(x_w, (\ell_w)^{1+\ep} )$.

Let us use this to estimate from above the value of  $\wM_1(I_w)$,  where  $w\in \Sigma_j$  is  the unique word such that $I_j(x)=I_w$. To do so, it is enough to consider the dyadic cubes $I_{w'}$ such that:
\begin{itemize}
\item
$|w'|\geq |w|=j$,
\item
$w' \in \mathcal{S}_{j'}(\eta)$,
\item
$x_{w'} \in \bigcup_{ I\in \mathcal{N}(I_j(x))}I$.
\end{itemize}
As noticed in the proof of   Lemma \ref{lem3},   it is  enough to consider the words $w'$ such that $j \leq |w'| \leq  j/(\eta-\ep_j)$. Moreover, for $j$ large enough, it is not possible that $w'\in \mathcal{S}_{j'}(\eta)$ satisfies $j'\leq j/(\eta+\ep)$ and $|x_{w'} - x| \leq 2\cdot 2^{-j}.$ This would imply that
$$|x_{w'} - x| \leq 2\cdot 2^{-j} \leq 2\cdot 2^{-j' (\eta+\ep)} < (\ell_{w'})^{1+\ep},$$ 
contradicting $\delta_x=1$. Hence, the value of $\M_1(I_w)$ is reached for some word $w'$ of length satisfying $j/(\eta+\ep) \leq j' \leq  j/(\eta-\ep_j)$. Using the same kind of estimates as in Lemma \ref{lem3} and the doubling property of $\mu$, the corresponding value $ \mu(I_{w'}) $ can   be bounded by
\begin{equation}
\label{adapt1}
\mu(I_w) 2^{-j\widetilde \ep_j}  \leq  \mu(I_{w'}) \leq    \mu(I_w) 2^{j\widetilde \ep_j} ,
\end{equation}
for some other sequence $\widetilde \ep_j$ independent of $x$ (we keep the same notation $(\widetilde \ep_j)$  for the new sequence, to keep them simple). Hence $\wM_1(I_w)$ satisfies
\begin{eqnarray}
\label{majmin2}
\max \left(   \mu(I) : {I\in \mathcal{N}(I_j(x))}\right ) \, 2^{-j\widetilde \ep_j} 
 \leq \wM_1(I_w)  \leq   \max \left(   \mu(I) : {I\in \mathcal{N}(I_j(x))}\right ) \, 2^{j\widetilde \ep_j}.
\end{eqnarray}
Taking the liminf in \eqref{majmin2}, we conclude the proof of Proposition \ref{prop2}.
\end{proof}

 \subsubsection{{\bf Some approximating sets; heterogeneous ubiquity}}
 \label{sec-approx1}
 
In order to better understand the relation between the local behavior of $\mu$ and the approximation rate, we introduce some approximating sets.

\begin{definition}
Let $H\geq 0$, $\de\geq 1$ and let $\widetilde \xi:=( \xi_j)_{j\geq 1}$ be a  positive sequence. A word  $w\in \Sigma^*$ is said to satisfy property $\mathcal{P}(H,\widetilde \xi)$  whenever
  \begin{equation}
  \label{doublesided-0}
   |I_w^\eta|^{H+\xi_{|w|}}   \leq    \mu( I_w^\eta) \leq  |I_w^\eta| ^{H-\xi_{|w|}} .
   \end{equation}
Set 
$$ \mathcal{A}_{H,\widetilde\xi,\delta} = \bigcap_{J \geq 1} \  \bigcup_{j \geq J} \ \bigcup_{w\in \mathcal{S}_{j}(\eta) :  \mbox{{\tiny \, $w$ satisfies $\mathcal{P}(H,\widetilde \xi)$}   } }  B \Big (x_w, (\ell_w)^{\de}  \Big),$$

\end{definition}

The set $ \mathcal{A}_{H,\widetilde\xi,\delta}$ contains those points $x$ which are approximated at a rate $\de$ by dyadic points $x_w$  satisfying the property that the measure $\mu$ has a controlled local H\"older behavior around $x_w$.

\begin{proposition}
\label{majexpo}
If  $\widetilde\xi$ converges to 0 and $x\in\mathcal{A}_{H,\widetilde\xi,\delta}$,  then  
$$\eta H \leq \dimi (\M_1,x) \leq \frac{H }{\min(\delta,1/\eta)}.$$
\end{proposition}
\begin{proof}
The left inequality is true for all $x$'s, and follows simply from \eqref{majmin0}.

Let $x\in\mathcal{A}_{H,\widetilde\xi,\delta}$. There exists  an infinite sequence of  words $(w_n)_{n\geq 1}$  such that
\begin{itemize}
\sk
\item
$|x-x_{w_n}| \leq (\ell_{w_n})^\delta=2^{- \delta \lfloor |w_n|(\eta-\ep_{|w_n|})\rfloor}$,
\sk
\item
$p_{w_n}=1$,

\sk
\item
$  2^{-\lfloor \eta |w_n|\rfloor  (H+\xi_{|w_n|} ) } \leq \mu(I_{w_n}^\eta) \leq 2\cdot 2^{-\lfloor \eta |w_n|\rfloor   (H-\xi_{|w_n|})}$.

\end{itemize}

The last property follows from  \eqref{doublesided-0}.

Set $\widetilde \delta = \min(\delta, 1/\eta)$. For every $n\ge 1$, let $J_n=\lfloor   \widetilde\delta \lfloor |w_n|(\eta-\ep_{|w_n|})\rfloor \rfloor  $. We are going to estimate $\wM_1 (I_{J_n}(x)) $ from below. To do so, as noticed in the proofs of Lemma      \ref{prop2} and Proposition \ref{lem3},   it is  enough to consider the words $w'$ such that $J_n \leq |w'| \leq  J_n/(\eta-\ep_{J_n})$. 

In particular, by construction the dyadic cube $I_{w_n}$ satisfies that $I_{w_n}\subset \bigcup_{I\in \mathcal{N}(I_{J_n}(x))} I$ and 
$$J_n \leq\widetilde\delta |w_n|(\eta-\ep_{|w_n|}) \leq  |w_n| \leq     \frac{J_n}{\widetilde \delta (\eta-\ep_{|w_n|})} \leq   \frac{J_n}{ \eta-\ep_{J_n}}  ,$$
 since $1 \leq \widetilde \delta \leq  1/\eta$. So, 
\begin{eqnarray*}
\wM_1 (I_{J_n}(x))  & \geq  & \mu(I_{w_n}^\eta) \geq 2^{-[\eta |w_n|] (H+\xi_{|w_n|} ) }   \geq C 2^{-  (H+\xi_{|w_n|})  \eta  \frac{J_n} {\widetilde \de (\eta-\ep_{|w_n|}) }} \geq C2^{-J_n \cdot    \frac{  (H+\xi_{|w_n|})}{  \widetilde \delta} \cdot \frac{1}{1-\ep_{|w_n|} /\eta}  }.
\end{eqnarray*}
Consequently, 
$$\liminf_{n\to +\infty} \frac{\log \wM_1(I_{J_n}(x))}{\log 2^{-J_n}} \leq \lim_{n\to +\infty }  \frac{ (H+\xi_{|w_n|}) }{  \widetilde \delta} \cdot \frac{1}{1-\ep_{|w_n|} /\eta}  =  \frac{H}{  \widetilde \delta}.$$
\end{proof}

A remarkable property is that the Hausdorff dimension of the sets $\mathcal{A}_{H,\widetilde\xi,\delta}$ can be bounded from  below using the notion of {\em heterogeneous ubiquity} developed in \cite{BSubiquity1,BSubiquity2}. Next theorem corresponds to Theorem 2.7 in \cite{BSubiquity2}, or in a weaker form to Theorem 1.2 in \cite{BSubiquity1}, in the case $\ro=1$.

\begin{theorem}
\label{thubi}
Let $\mu$ be  a Gibbs capacity as in Definition \ref{defgibbscap}. With probability 1, since  \eqref{cover1} holds true, for all  $\de\ge 1$ and  $H\geq 0$  such that $\sigma_\mu(H)> 0$, there exists a sequence $\widetilde \xi$ decreasing to zero and a Borel probability measure $m_{H,\delta}$ such that:
\begin{itemize}
\item
$m_{H,\de}(\mathcal{A}_{H,\widetilde\xi,\delta}) >0$,
\sk
\item
for any set $E$ of Hausdorff dimension less than $\sigma_\mu(H)/\de$, one has $m_{H,\de}(E)=0$.
\end{itemize}
In other words, the lower Hausdorff dimension of $m_{H,\de}$ is larger than or equal to $\sigma_\mu(H)/\de$. In particular, $\dim \mathcal{A}_{H,\widetilde\xi,\delta}\geq \sigma_\mu(H)/\de$.
\end{theorem}

\begin{remark}
\label{formequiv}
(1) Actually the theorems of \cite{BSubiquity1,BSubiquity2}  are proved when the condition \eqref{doublesided-0} is replaced by 
\begin{equation}
  \label{doublesided2}
(\ell_w)^{H+\xi_{|w|}}   \leq    \mu(  B(x_w, \ell_w)) \leq  (\ell_w) ^{H-\xi_{|w|}} .
   \end{equation}
One easily checks that in the proofs of  \cite{BSubiquity1,BSubiquity2},   \eqref{doublesided2} can be replaced by  \eqref{doublesided-0}. Indeed, recalling \eqref{deflw}, the doubling property  of the Gibbs capacity $\mu$ implies that $ \mu(  B(x_w, \ell_w))$ and $\mu(I_w^\eta)$  are   comparable, since both $B(x_w, \ell_w)$ and $I_w^\eta$  contain $x_w$ and have comparable diameters. 

(2) Theorem  2.7 of \cite{BSubiquity2} deals with $H\in(H_{\min},H_{\max})$, but a simple adaptation of the proof makes it possible to extend the result to $H\in \{H_{\min},H_{\max}\}$.
\end{remark}

\subsection{Upper bound for the spectrum of $\M_1$}
 
To find an upper bound for $\si_{\M_1}$, we use the inequality \eqref{formalism-1}. Recall that  ${q_1}$ is the unique solution of the equation
$\tau_\mu(q)=0$.

\begin{proposition}
\label{propupper}
Almost surely, for every $q\in \R$, we have
\begin{equation}
\label{majtau}
\tau_{\M_1}  (q)  \geq  T(q):=\begin{cases}  \   \tau_{\mu}(q)  & \mbox{ if $q<{q_1}$},\\
\  \eta\tau_{\mu}(q) & \mbox{ if $q\ge {q_1}$}.
\end{cases}
\end{equation}
\end{proposition}
\begin{proof}

Let $j\geq 1$. Recalling \eqref{deftau} and \eqref{defleaderM}, one has to estimate the sum
$ \sum_{w\in \Sigma_j} \wM_1( I_w)^q$.

\sk {\bf $\bullet$  Case   $q<0$:} using the lower bound in \eqref{majmin0},   one has
\begin{eqnarray*}
 \sum_{w\in \Sigma_j} \wM_1(  I_w)^q &  \leq  &    \sum_{w\in \Sigma_j} \max(\M_1(I): I\in \mathcal{N}(I_j(x)))^q \\
 & \leq&   3^{d+1} \sum_{w\in \Sigma_j}   \mu(I_w)^{q} 2^{-q|w|\widetilde \ep_{|w|} }   \leq 3^{d+1} \cdot  2^{- j(\tau_{\mu,j}(q) +q\widetilde\ep_j) }.
 \end{eqnarray*}
Since $\lim_{j\to +\infty}  \widetilde\ep_j =0$, letting $j$ go to infinity yields  $\tau_{\M_1}(q) \geq \tau_{\mu}(q)$.

\sk {\bf $\bullet$  Case $q>0$:}  the upper bound in \eqref{majmin0} is actually  too crude. Using the covering property \ref{lem1}, one knows that for every word $w\in \Sigma_j$, the value of $\wM_1(I_w)$ is reached for one word $w' $ whose length is between $j$ and $j/(\eta-\ep_j)$. Conversely, each word $w'  \in \mathcal{S}_{j'}(\eta)$ with $j  \leq j' \leq j/(\eta-\ep_j)$ may contribute to at most $3^{d}$ values of $\wM_1(I_w)$ for  $|w|=j$. Hence, one has
\begin{eqnarray*}
 \sum_{w\in \Sigma_j} \ \wM_1 (I_w)^q &  \leq  & 3^{d}  \sum_{j'=j}^{\lfloor j/(\eta-\ep_j)\rfloor } \sum_{w\in \mathcal{S}_{j'}(\eta) }  \mu(I_w^\eta)^q  \leq 3^{d}   \sum_{j'=j}^{\lfloor j/(\eta-\ep_j)\rfloor } \sum_{w\in\Sigma_{j'}}   p_w\mu(I^\eta_w) ^{q}
   \end{eqnarray*}
Recalling that when $w'\in \Sigma_{j'}$,  $p_{w'}$ is a Bernoulli variable with parameter $2^{-dj'(1-\eta)}$, one gets
$$
\mathbb{E} \Big ( \sum_{w\in \Sigma_j} \  \wM_1 (I_w)^q\Big )\le 3^{d}   \sum_{j'=j}^{\lfloor j/(\eta-\ep_j)\rfloor } \sum_{w\in\Sigma_{j'}}   2^{-dj'(1-\eta)}\mu(I^\eta_w) ^{q}.
$$
Now, observe that each cube $I_{w'}$, for $w'\in \Sigma_{\lfloor \eta j'\rfloor }$ and $j  \leq j' \leq j/(\eta-\ep_j)$, contains    $   2^{d(j'-\lfloor j'\eta\rfloor) }$  dyadic cubes   of generation $2^{j'}$. In other words, each $w'\in \Sigma_{\lfloor \eta j'\rfloor }$ with  $j  \leq j' \leq j/(\eta-\ep_j)$ appears  $  2^{d(j'-\lfloor j'\eta\rfloor )}$  times in the previous double sum as a cube of the form  $I_{w}^\eta$ with $w\in\Sigma_{j'}$. Thus,
\begin{eqnarray}
\label{malE1}
 \mathbb{E} \Big ( \sum_{w\in \Sigma_j}  \wM_1 (I_w)^q\Big )&  \leq  &      C  \sum_{j'=j}^{\lfloor j/(\eta-\ep_j)\rfloor } \sum_{w\in \Sigma_{[j'\eta]}}    \mu(I_w)^{q}  \leq C  \sum_{j'=j}^{\lfloor j/(\eta-\ep_j)\rfloor } 2^{- {\lfloor j'\eta\rfloor }\tau_{\mu,{[j'\eta]}}(q)  },
 \end{eqnarray}
for some constant $C$ depending only on the dimension $d$. The behavior of the last sum depends on the value of $\tau_{\mu}(q) $. 

\sk

\sk {\bf $\bullet$  Subcase $q>0$,  $\tau_{\mu}(q) <0$:} this coincides with the fact that  $q<{q_1}$. Fix  $s<\tau_{\mu}(q) $ and $\ep = (\tau_\mu(q)-s)/3$.  Recalling \eqref{lim-tau}, for $j'$ sufficiently large, $|\tau_{\mu,[j'\eta]}(q)  - \tau_\mu(q)|\leq \ep$. Then, since $\tau_\mu(q)<0$, one has 
$$
 \mathbb{E} \Big ( \sum_{w\in \Sigma_j}  \wM_1 (I_w)^q\Big )   \leq  C  \sum_{j'=j}^{\lfloor j/(\eta-\ep_j)\rfloor } 2^{- {\lfloor j'\eta\rfloor }( \tau_\mu(q) - \ep)  } \leq C 2^{-  j\eta/(\eta-\ep_j) ( \tau_\mu(q) - \ep)  } .
$$
Thus, 
$$\mathbb{E} \Big (\sum_{j\ge 1} 2^{js} \sum_{w\in \Sigma_j} \ \wM_1(I_w)^q\Big ) \leq  C \sum_{j\ge 1} 2^{j(s-  \eta/(\eta-\ep_j) ( \tau_\mu(q) - \ep) ) }  <\infty,$$ 
since when $j$ becomes large, $s-  \eta/(\eta-\ep_j) ( \tau_\mu(q) - \ep)  \leq  -\ep$. Hence, almost surely,
$$\sum_{j\ge 1} 2^{js} \sum_{w\in \Sigma_j} \ \wM_1(I_w)^q  = \sum_{j\ge 1} 2^{js}  2^{-j\tau_{\M1,j}(q)}<\infty.$$
 Consequently, $\tau_{\M_1}(q) \geq  s$, and letting $s$ tend to $\tau_\mu(q)$ one gets $\tau_{\M_1}(q) \ge \tau_{\mu}(q) $ almost surely. 

\sk

\sk {\bf $\bullet$  Subcase $q>0$,  $\tau_{\mu}(q) >0$:} this corresponds to the values $q> {q_1}$. The   arguments   above can be adapted here. Fix  $s<\tau_{\mu}(q) $,  $\ep = (\tau_\mu(q)-s)/3$ and $j'$ large.  Since $\tau_\mu(q)>0$, \eqref{malE1} gives in this case (since it involves a sum of a geometric sequence)
 $$
 \mathbb{E} \Big ( \sum_{w\in \Sigma_j}  \wM_1 (I_w)^q\Big )     \leq C  2^{-  j\eta ( \tau_\mu(q) - \ep)  } .
$$

 Letting $j$ tend to infinity,  and then $\ep$ to zero, we obtain  $\tau_{\M_1}(q) \geq  \eta\tau_{\mu}(q)$, almost surely.

\sk {\bf $\bullet$  Subcase $q>0$,  $\tau_{\mu}(q) =0$:}  In this case, a simple calculation gives  $\tau_{\M_1}(q) \geq  \tau_{\mu}(q)$.

\sk

Finally, the previous lower bounds hold for each   $q$ almost surely, but $\tau_{\M_1}$ and $\tau_\mu$ being concave functions, they hold   almost surely for all    $q$.
\end{proof}
\begin{corollary}
\label{cor-upper1}
With probability 1, one has $\sigma_{\M_{1}} \leq \tau^*_{\M_{1}}\le T^*$.
\end{corollary}
\begin{proof}
This is a direct consequence of the multifractal formalism  \eqref{formalism-1}, i.e. $\sigma_{\M_{1}}\le \tau^*_{\M_{1}}$, together with the inequality $\tau_{\M_{1}}\ge T$ established in the previous Proposition \ref{propupper}, so that $\tau^*_{\M_{1}}\le T^*$.
\end{proof}

Simple calculations show that the Legendre transform   of $T$ coincides with the formula \eqref{spec-M1}, i.e. 
\begin{equation}
\label{form1}
T^*(H) = \begin{cases} \      \eta \sigma_\mu(H/\eta) &  \mbox{ if }  \ \eta H_{\min}\le  H < \eta  {H_1} \\ 
\ \   \ \  {q_1} H & \mbox{ if }  \ \eta {H_1} \leq  H< {H_1}  \\   
    \ \ \ \sigma_\mu (H)  & \mbox{ if }   {H_1}\le H<® H_{\max}\end{cases}
\end{equation}
and $T^*(H)=-\infty$ otherwise. It remains to prove that $ \sigma_{M_{1}(\mu)}\ge T^*$.

 \subsection{Lower bound for the singularity spectrum of $\M_1$}

We distinguish the  three regions given by formula \eqref{spec-M1}
 that we claim to hold for $\si_{\M_1}$. Recall that  ${q_1}$ is the unique solution of the equation
$\tau_\mu(q)=0$,  ${H_1}= \tau_\mu'({q_1})$, and that this implies that $ {q_1}=\sigma_\mu'({H_1})$.

 \sk
 
 Observe first that by \eqref{majmin0}, necessarily $\dimi(\M_1,x) \geq \eta H_{\min} $ for every $x\in \zu^d$.
 
 \subsubsection{{\bf Case $\eta H_{\min} \leq H\leq \eta {H_1}$}}

Observe that we are in the increasing part of the spectrum. Heuristically, the value of the Hausdorff dimension of $\underline E_{\M_1}(H) = \{x:\dimi (\M_1,x)=H\}$ is obtained for the points $x$   having an approximation rate equal to $1/\eta$ (recall Definition \ref{def-approx}).

\mk

Let $\eta H_{\min} \leq H\leq \eta {H_1}$ and  set $\widetilde H= H/\eta$. If $\widetilde H=H_{\min}$ and $\sigma_\mu(H_{\min})=0$, since it is easy to construct a decreasing sequence 
$\widetilde\xi$ such that $\mathcal{A}_{H_{\min},1/\eta,\widetilde\xi}\neq\emptyset$, and for all $x\in \mathcal{A}_{H_{\min},1/\eta,\widetilde\xi}$ one has both $\dimi(\M_1,x) \geq \eta H_{\min}$ by the above observation and $\dimi(\M_1,x) \leq \eta H_{\min}\le H_{\min}$ by Proposition~\ref{majexpo}, we get $\underline E_{\M_1}(\eta H_{\min})\neq\emptyset$ hence $\dim \sigma_{\M_1}(\eta H_{\min})\ge 0=\sigma_\mu(H_{\min})=T^*(\eta H_{\min})$.

If $\widetilde H\neq H_{\min}$ or $\sigma_\mu(H_{\min})>0$,  by Theorem \ref{thubi} applied with $\widetilde H$ and $\de=1/\eta$, there exists a sequence $\widetilde \xi$ and a measure $m_{\widetilde H, 1/\eta}$ with the properties:
 \begin{itemize}
 \item
$ m_{\widetilde H, 1/\eta}(\mathcal{A}_{\widetilde H,1/\eta,\widetilde\xi} )>0$,
\item
$m_{\widetilde H,1/\eta}(E)=0 $  for any  Borel set $E$ such that $\dim E < \sigma_\mu(\widetilde H)/(1/\eta)= \eta \sigma_\mu(\widetilde H)$.
\end{itemize}
 
 Applying Proposition \ref{majexpo}, every point $x \in  \mathcal{A}_{\widetilde H,1/\eta,\widetilde\xi} $ satisfies
  $$\dimi (\M_1,x) \leq \frac{\widetilde H}{\min(1/\eta,1/\eta)} = \eta\widetilde H=H.$$
  
  We deduce (recall Definition~\ref{defnewsets}) that $m_{\widetilde H,1/\eta} \left(\underline{E}^{\le}_{\M_1}(H)\right )>0.$

  Further, note that the inequality $\tau_{\M_1}\ge T$ implies that $\tau_{\M_1}'(0^+)\ge T'(0)$.  Consequently, since $H\le T'(0)=\tau_\mu'(0)$,  
  one deduces from \eqref{upcap1} that for every $n\geq 1$,
  \begin{align*}
  \dim  \underline{E}^{\le}_{\M_1}(H-1/n)  \leq    (\tau_{\M_1} )^*(H-1/n)   \leq    \eta \sigma_\mu\left(\frac{H-1/n}{\eta} \right)
   =  \eta \sigma_\mu\left( \widetilde H  - \frac{1}{n{\eta}} \right)  <  \eta  \sigma_\mu( \widetilde H ).
  \end{align*} 
  Hence $m_{\widetilde H,1/\eta} \left(\underline {E}^{\le}_{\M_1}(H-1/n)\right )=0$. Finally, since 
  \begin{equation}
  \label{decomp}
   \underline {E}^{\le}_{\M_1}(H)=\underline E_{\M_1}(H) \cup \bigcup_{n\geq 1} \underline {E}^{\le}_{\M_1}(H-1/n)),
   \end{equation}
   one obtains $m_{\widetilde H,1/\eta} (\underline E_{\M_1}(H) )>0$, which implies the expected lower bound $\sigma_{\M_1} (H)  \geq \eta  \sigma_\mu( \widetilde H )$.  
  
 \subsubsection{{\bf Case $\eta {H_1}<  h \le  {H_1}$}}

 The proof is similar to that used  for $ \eta H_{\min} \leq H\le \eta {H_1}$, except that here, the value of the Hausdorff dimension of $\underline{E}_{\M_1}(H)$ is obtained for points $x$ which have an approximation rate between $1 $ and  $1/\eta$.

\mk

Let $\eta {H_1}<  H \le   {H_1}$ and $\de=\frac{{H_1}}{H} \in [1,1/\eta)$. One can apply Theorem \ref{thubi} with ${H_1}$ and $\de$. We reproduce   exactly the same arguments  as in the previous subsection. There exists a sequence $\widetilde \xi$ and a measure $m_{H_1, \de}$ with the properties:
 \begin{itemize}
 \item
$m_{H_1, \de}(\mathcal{A}_{H_1,\de,\widetilde\xi} )>0$,
\item
$m_{H_1,\de}(E)=0 $  for any Borel set $E$ such that  
$$\dim E < \frac{\sigma_\mu({H_1})}{\de}  =  \frac{ \sigma_\mu({H_1})}{    {H_1} } H   = {q_1} H,$$
where the last equality comes from the equalities $\tau_\mu({q_1})=0$, ${H_1}=\tau_\mu'({q_1})$ and  $ \sigma_\mu(\tau_\mu'(q))=\tau_\mu^*(\tau_\mu'(q))=q\tau_\mu'(q)-\tau_\mu(q)$ for all $q\in\mathbb R$.
\end{itemize}
 
 \sk
 
 Further, applying Proposition \ref{majexpo}, every point $x \in  \mathcal{A}_{{H_1} ,\de,\widetilde\xi} $ satisfies
  $$ \dimi (\M_1,x) \leq \frac{{H_1}}{\min(\de,1/\eta)} =\frac{{H_1}}{\de}= \frac{{H_1}}{{H_1}/H}= H  .$$
  Consequently, $
  m_{{H_1}, \de} \left(\underline {E}^{\le}_{\M_1}(H)\right )>0$, and by the same argument as above,  
  \begin{eqnarray*}
  \dim \underline {E}^{\le}_{\M_1}(H-1/n)   \leq    (\tau_{\M_1})^*(H-1/n)   \leq   {q_1}(H-1/n),
    \end{eqnarray*} 
  hence $m_{{H_1},\delta} \left(\underline {E}^{\le}_{\M_1}(H-1/n)\right )=0$. Finally, using \eqref{decomp},   necessarily $m_{H_1,\delta} (\underline {E}_{\M_1}(H))>0$, hence $\sigma_{\M_1}(H)  \geq {q_1} H$,   
which again is the expected lower bound.

 \subsubsection{{\bf Case $  H\geq H_{1}$}} We will need the following proposition.  Recall that a Borel probability measure $\nu$ on $\R^d$ is said to be exact dimensional with dimension $D\ge 0$ if $\nu(E_\nu(D))=1$, and that in this case $\nu(E)=0$ for all Borel sets such that $\dim E>D$. 
\begin{proposition}\cite[Theorem 5(1)]{BS2020}\label{proplast} With probability 1, for all $H\in [H_{\min},H_{\max}]$, there exists an exact dimensional Borel probability measure $\mu_H$ of dimension $\sigma_\mu(H)$ such that $$\mu_{H} \left(\{x\in[0,1]^d: \, \dim(\mu,x)=H\} \bigcap\{x\in[0,1]^d:\, \delta_x=1\}\right)  = \mu_{H} \left( \{x\in[0,1]^d:\,   \dim(\mu,x)= H\}\right).$$
\end{proposition}
Let us mention that if $\mu\approx\nu^\gamma$ where $\nu$ is a Gibbs measure associated with a H\"older potential $\varphi$, then for $H\in(H_{\min},H_{\max})$, $H=\tau_\mu'(q)$ for some $q$ and $\mu_H$ can be taken equal to the Gibbs measure associated with the potential $q\gamma\varphi$. 

\sk

Now, if $H\ge H_1$, by Proposition \ref{prop2}, any $z\in \{x\in[0,1]^d:\,  \dim(\mu,x)=H\} \bigcap\{x\in[0,1]^d:\, \delta_x=1\}$ satisfies $\dimi(\M_1,z) =  H $.
Hence $\mu_H$ is supported on $\underline E_{\Mr}(H)$, so $\sigma_{\M_1}(H) \geq \sigma_\mu( H)$, which is the expected lower bound.

\section{Proof of Theorem~\ref{main2}} \label{pfth2}

Here, $\ro\in  (0,1)$ and $\M_{\ro,\eta}$ is simply denoted  by $\Mr$. 

Since $\ro <1$,   there is a lot of  redundancy since the dilation parameter over-compensates the sampling.  

Recalling Definition \ref{defMetarho}, each $\Mr (I)$, for $I\in \mathcal{D}_j$, can be written as
 $$\Mr(I) = \sup_{I'\subset I, \ p_{I'}=1} \mu((I')^\reta),$$ 
and let us introduce by analogy with \eqref{defleaderM}
\begin{equation}
\label{defleaderMr}
\wMr(I) = \max(\Mr(I): I\in \mathcal{N}(I_j(x))).
\end{equation}
Computing $\underline \dim(\Mr,x)$ amounts to estimating the behavior of $\wMr(I_j(x))$ when $j$ tends to infinity,.

\subsection{Some  auxiliary notations and properties}
\label{sec5-1}

We omit the subscript $\mu$ in the following definitions, although the quantities depend on $\si_\mu$.
\begin{definition}
For all $H\in [H_{\min},\tau'(0)]$ and $\delta \ge 1$, define 
\begin{align}
\label{defdeltarho} \delta_\varrho(H)& = \frac{d(1-\varrho)+\varrho\sigma_\mu(H)}{\sigma_\mu(H)}\\
\label{defthetarho}
 \theta_\ro(H) & =    \frac{\varrho H}{\delta_\ro (H)}.
 \end{align}
 \end{definition}
 
 Observe that: 
 \begin{itemize}
 \item
The mapping $H\mapsto \delta_\ro(H)$ decreases from $\frac{d(1-\varrho)+\varrho\sigma_\mu(H_{\min})}{\sigma_\mu(H_{\min})}$ to $1$ over  $[H_{\min},\tau_\mu'(0)]$ (in particular, $\delta_\ro(H)\geq 1$).
\item
The mapping $H\mapsto \theta_\ro(H)$  increases continuously from $\frac{\varrho H_{\min}\sigma_\mu(H_{\min})} {d(1-\varrho)+\varrho\sigma_\mu(H_{\min})} $ to $\varrho \tau_\mu'(0)$.
\item
Since $\delta_\ro(H)\geq 1$, $\theta_\ro(H) \leq \ro H$, with equality only if $\delta_\ro(H)=1$, i.e. $H=\tau_\mu'(0)$.
\end{itemize}

Additional parameters are needed (they were defined in the introduction, but we give more details now):
\begin{itemize}
 \sk
 \item 
Let   ${q_\ro}$ be  the unique $q$  such that 
\begin{equation}
\label{deftaurho}
\tau_\mu({q_\ro})=\frac{d(1-\varrho)}{\varrho},
\end{equation}
 and set 
 $${H_\ro}=\tau_\mu'({q_\ro}).$$
  Recall that this ${q_\ro}$ exists since $\tau_\mu$ is an homeomorphism of $\R$.
 
 \noindent 
The mapping 
\begin{equation}
\label{funcro}
H\in[H_{\min},H_{\max}]\mapsto \frac{d(1-\varrho)+\varrho \sigma_\mu(H)}{H}
\end{equation}  is increasing on $[H_{\min},{H_\ro}]$ and   decreasing on $[{H_\ro}, H_{\max}]$. Indeed, the derivative of this mapping   vanishes when $\ro H \si'_\mu(H) - d(1-\ro)-\ro \si_\mu(H)=0$, i.e. $ H \si'_\mu(H) -    \si_\mu(H) = \frac{d(1-\varrho)}{\varrho}$. The fact that $\tau_\mu=\si_\mu^*$ allows to conclude.   So the mapping reaches its maximum at  ${H_\ro}$.

 \sk
 \item 
A direct calculation  shows that 
\begin{equation}
\label{formula-rho}
 \frac{d(1-\varrho)+\varrho \sigma_\mu({H_\ro})}{\varrho {H_\ro}}=\frac{\sigma_\mu({H_\ro})}{\theta_\ro ({H_\ro})}={q_\ro}.
\end{equation}
Indeed, the first equality comes from   \eqref{defdeltarho}  and \eqref{defthetarho}. The second one follows from \eqref{deftaurho} and the fact that, by Legendre transform, $\sigma_\mu({H_\ro}) = {q_\ro} {H_\ro}- \tau_\mu({q_\ro}) $.
 \end{itemize}

 Next, as explained in Definition \ref{defDetarho}, three cases must be distinguished depending on the behavior of the multifractal spectrum $\si_\mu$ of $\mu$.
 
 \begin{itemize}
 \sk
 \item
 
{\bf Case 1.   $\sigma_\mu(H_{\min}) > \frac{d(1-\varrho)}{1/\eta-\varrho}$:} 
For every  $H\in [H_{\min},\tau_\mu'(0)]$,  $\sigma_\mu(H)>\frac{d(1-\varrho)}{1/\eta-\varrho}$, and $\delta_\ro(H) \leq \delta_\ro(H_{\min})= \frac{d(1-\varrho)+\varrho\sigma_\mu(H_{\min})}{\sigma_\mu(H_{\min})}    < 1/\eta$. In particular, $\ro\eta H_\ro < \theta_\ro(H_\ro)$.

 Notice that $\frac{d(1-\varrho)}{1/\eta-\varrho}$  is always smaller than $d$, and that it decreases from $\eta d$ to $0$ as $\varrho$ increases from $0$ to $1$.

As a conclusion, 
\begin{align}
\label{ordre1}
 \varrho \eta H_{\min}<  \varrho \eta {H_\ro} < \thetar ({H_\ro}) <\ro H_\ro <\varrho\tau_\mu'(0)\le \varrho H_{\max},
 \end{align}
 which explains the regions in formula \eqref{spec1}.
 \sk
 \item

{\bf Case 2.    $\sigma_\mu(H_{\min})\le  \frac{d(1-\varrho)}{1/\eta-\varrho}$:}    By continuity of $\si_\mu$ in its increasing part, there exists  a (necessarily unique) $H=H_{\ro,\eta}\in [H_{\min},\tau_\mu'(0)]$ such that 
\begin{equation}
\label{defHmurho}
\sigma_\mu(H_{\ro,\eta})=\frac{d(1-\varrho)}{1/\eta-\varrho}.
\end{equation}
This is equivalently written as 
$$\delta_\ro(H_{\ro,\eta})=1/\eta.$$
Note that $H_{\ro,\eta}=H_\ro$ if and only if $q_\ro\tau_\mu'(q_\ro)=d(1-\varrho)/(1-\ro\eta)$. 
\sk

 \begin{itemize}
 \sk
 \item {\bf Case 2.a. $H_{\ro,\eta}< {H_\ro}$:}  here $\delta_\ro ({H_\ro})< 1/\eta$, so  $\varrho\eta H_{\ro,\eta}< \varrho\eta {H_\ro} <  \theta_\ro(H_\varrho)$, and 
\begin{align}
\label{ordre2}
\varrho \eta H_{\min} <  \varrho \eta H_{\ro,\eta} <  \varrho \eta {H_\ro} < \thetar ({H_\ro}) <\ro H_\ro< \varrho\tau_\mu'(0) 
 \le \varrho H_{\max},
 \end{align}
  which explains the   regions in formula \eqref{spec2}.

\sk\item 
{\bf Case 2.b. $H_{\ro,\eta} \geq {H_\ro}$:} now   $\delta_\ro ({H_\ro})\geq 1/\eta$, so $  \theta_\ro(H_\ro) \le  \varrho\eta {H_\ro} \le  \varrho\eta H_{\ro,\eta}$, and the proof will show that only the value of $\varrho\eta H_{\ro,\eta}$ matters. Hence, the relevant regions in \eqref{spec3} are
\begin{align}
\label{ordre3}
\varrho \eta H_{\min} <  \varrho\eta H_{\ro,\eta} < \varrho\tau_\mu'(0)  \le \varrho H_{\max} .
\end{align}

\end{itemize}\end{itemize}

One has the following simple rewriting of \eqref{defdeltarho}:
\begin{equation}
\label{rewriting}
\displaystyle\frac{d(1-\varrho)+\varrho \sigma_\mu(\theta_\ro^{-1}(H))}{\delta_\ro (\theta_\ro^{-1}(H))}=\sigma_\mu(\theta_\ro^{-1}(H)) .
\end{equation}
All these formulas will be useful when looking for the optimal values of the parameters to compute the Hausdorff dimensions of the level sets $E_{\Mr}(H)$ in the corresponding intervals of values for $H$.

\mk

We conclude this section with the proof of Proposition~\ref{regularityDetarho}.

\begin{proof}[Proof of Proposition~\ref{regularityDetarho}]The regularity properties (continuity, analyticity on the region's interior) of $D_{\mu,\ro,\eta}$ are easily checked. It is also immediate that among the intervals distinguished in the formulas \eqref{spec1}, \eqref{spec2} and \eqref{spec3} providing $D_{\mu,\ro,\eta}$, the only one over which the concavity of $D_{\mu,\ro,\eta}$ is not clear is the interval where $\sigma_\mu(\theta_\ro^{-1}(H))$. Once this concavity property is verified, the global concavity of $D_{\mu,\ro,\eta}$ is immediate for the Cases 1 and $2.b.$, and it follows from the fact that $D_{\mu,\ro,\eta}'(\ro\eta H_{\ro,\eta}^-)=\eta^{-1}D_{\mu,\ro,\eta}'(\ro\eta H_{\ro,\eta}^+)>D_{\mu,\ro,\eta}'(\ro\eta H_{\ro,\eta}^+)$ in Case $2.a$. 

To prove that    $H\mapsto\sigma_\mu(\theta_\ro^{-1}(H))$ is concave over $[H_\varrho,\ro\tau_\mu'(0)]$ (and so over $[H_{\varrho,\eta},\ro\tau_\mu'(0)]$ in Case 2.b.), one rather shows that   its reciprocal map $F=\theta_\ro\circ\sigma_\mu^{-1}$ is convex over $[\sigma_\mu(H_\ro),d]$. Indeed, one sees that $F(u)=\varrho\frac{u\sigma_\mu^{-1}(u)}{d(1-\ro)+\ro u}$, and a calculation shows that 
$$
F''(u)=\frac{\ro}{(d(1-\ro)+\ro u)^3}(G(u)+2d(1-\ro)h(u)),
$$
where, setting $g=\sigma_\mu^{-1}$, $h(u)=g'(u)(d(1-\ro)+\varrho u)-\ro g(u)$, and $G$ is a polynomial function with positive coefficients in the variables $u,g(u),g'(u),g''(u)$, which are all non-negative over $[\sigma_\mu(H_\ro),d]$ (using the concavity of $\si_\mu$). Moreover, setting $H=g(u)$, one has $g'(u)=\si_\mu'(H)^{-1}$. So $h(u)\ge 0$ if and only if $d(1-\ro)\ge \ro (\sigma_\mu'(H)H-\sigma_\mu(H))$, that is $H=\tau_\mu'(q)$ with $\tau_\mu(q)\le d(1-\ro)/\ro)$, i.e. $H\ge H_\ro$. Consequently, $F''$ is positive over $[\sigma_\mu(H_\ro),d]$, hence the desired result. 
\end{proof}

\subsection{Rewriting the formula for $D_{\mu,\ro,\eta}$}

Let us define the mapping (again, we omit the dependence on $\mu$ to keep the notations   as light as possible)
\begin{align}
\label{defmrho}
m_\varrho(H,\delta)& = \min\Big  (\frac{d(1-\varrho)+\varrho\sigma_\mu(H)}{\delta},\sigma_\mu(H) \Big ), \ H\in\R,\, \delta\ge 1.
 \end{align}

The main proposition of this section,  key to compute the multifractal spectrum of $\Mr$, is the following alternative formula for $D_{\mu,\ro,\eta}$. This formula  reflects the intuition that     the fact that $\Mr$ has at $x$ a local dimension equal to $H$,  results from a compromise between the local behavior of $\mu$ at $x$ and the approximation rate of $x$ by the surviving vertices.

\begin{proposition}\label{pro5.1}
If $H\ge 0$, set
\begin{equation}
\label{defdtilde}
\widetilde D_{\mu,\ro,\eta}(H)=\max\Big \{ m_\ro(u,\delta) : 0\leq  \frac{\varrho u}{\delta}\le H,\  \displaystyle 1\le \delta\le 1/\eta\Big \}. 
\end{equation}
One has $ D_{\mu,\ro,\eta}=\widetilde D_{\mu,\ro,\eta}$ over the interval $[0, \varrho \tau_\mu'(0)]$.
\end{proposition}
\begin{proof} 

Observe that when $u<H_{\min}$, $\si_\mu(u)=-\infty$,  so $m_\varrho(u,\delta)=-\infty$.  So it is enough to consider  $u\in [H_{\min},H_{\max}]$.

\sk

{\bf Fact:} When $\delta=\delta_\ro(u)$ and $u\leq \ro\tau_\mu'(0)$ then $\si_\mu(u)=\frac{d(1-\varrho)+\varrho\sigma_\mu(u)}{\delta}$. When $\delta<\delta_\ro(u)$, the minimum in \eqref{defmrho}  with respect to $u$ is reached by $\si_\mu(u)$, while when $\delta\geq \delta_\ro(u)$ it is reached by $\frac{d(1-\varrho)+\varrho\sigma_\mu(u)}{\delta}$.

\mk

 {\bf Case 1. $\sigma_\mu(H_{\min}) > \frac{d(1-\varrho)}{1/\eta-\varrho}$:} 
  Let $H\in [0,\tau_\mu'(0)]$. 

For all $u\in [H_{\min},\tau_\mu'(0)]$,  $\sigma_\mu(u)\geq \sigma_\mu(H_{\min}) > \frac{d(1-\varrho)}{1/\eta-\varrho}$, so 
  $$\delta_\ro(u)  = \ro +\frac{d(1-\ro)}{\si_\mu(u)}<1/\eta .$$
  
(i) When $1\le \delta\le \delta_\ro (u)$ is fixed,  by the fact above, the minimum   in \eqref{defmrho}  is reached  at $\si_\mu(u)$. So,  when looking  for the maximum in \eqref{defdtilde}, one tries to maximize $\si_\mu(u)$ for  $H_{\min} \leq  \frac{\varrho u}{\delta}\le H,\  \displaystyle 1\le \delta\le \delta_\ro(u)$.  
  This maximum is reached  for some $u$  when $\delta=\delta_\ro(u)$, i.e. one  has to maximize $\si_\mu(u)$ for  $H_{\min} \leq  \frac{\varrho u}{\delta_\ro(u)}\le H$.   The maximum is reached when $ \frac{\varrho u}{\delta_\ro(u)} =H$, i.e. when $u=\theta_\ro^{-1}(H)$ (the reciprocal image of $H$ by $\theta_\ro$), and its value is $\si_\mu( \theta_\ro^{-1}(H))$.

  \mk
  
(ii)  Now, when   $ \delta_\ro (u) \leq \delta \leq 1/\eta$, the minimum   in \eqref{defmrho}  is reached  at $\frac{d(1-\varrho)+\varrho\sigma_\mu(u)}{\delta}$. So one has to maximize $\frac{d(1-\varrho)+\varrho\sigma_\mu(u)}{\delta} $ for  $H_{\min} \leq  \frac{\varrho u}{\delta}\le H,\  \displaystyle  \delta_\ro(u) \le \delta\le 1/\eta $.      Given $\delta$, this is maximal when $\varrho u/\delta=H$, and it equals $H\frac{d(1-\varrho)+\varrho\sigma_\mu(u)}{\ro u}$. Since $ \delta_\ro (u) \leq \delta$, this  implies that $\ro u/\delta_\ro(u) \geq H$, i.e. $u\geq \theta_\ro^{-1}(H)$. In addition, since $\delta \leq 1/\eta$, the optimal $u$ satisfies  $ u=H/(\ro\eta)$. As a conclusion,  the maximum in  \eqref{defdtilde} in this range of $\delta$'s  can be written 
$$\max\left\{H\frac{d(1-\varrho)+\varrho\sigma_\mu(u))}{\varrho u}: \theta_\ro ^{-1}(H)\le u\le H/\varrho\eta\right\}.$$ 
Recall that    the mapping  $u\mapsto \frac{d(1-\varrho)+\varrho\sigma_\mu(u))}{ u}$ increases over $[H_{\min},{H_\ro}]$ and decreases over $[{H_\ro},\tau_\mu'(0)]$.  Hence, the value of the above maximum gives birth to three phases, depending on the position of $H_\ro$ with respect to the interval $[\theta_\ro ^{-1}(H), H/\varrho\eta]$ (or by symmetry, on the position of $H$ with respect to the interval $[\ro\eta H_\ro,\theta_\ro(H_\ro)]$).   
 
 More precisely:
 \begin{itemize}
 \item when $ \ro\eta H_{\min} \leq H\le \varrho\eta {H_\ro}$, the maximum  is reached when  $u=H/\varrho\eta$ (the corresponding  $\delta$ is $1/\eta$), and it equals $\eta (d(1-\ro)    +\varrho\sigma_\mu(H/\varrho\eta))$.
 \item when $   \varrho\eta {H_\ro}\le H\le \theta_\ro ({H_\ro}) $, it is reached at $u={H_\ro}$ (with $\delta= \delta_\ro({H_\ro}){H_\ro}/H$), and it equals $H\frac{d(1-\ro) +\varrho \sigma_\mu({H_\ro})}{\varrho {H_\ro}}=H \frac{\sigma_\mu(\theta_\ro ({H_\ro}))}{\theta_\ro(H_\varrho)}$ (here, we used \eqref{formula-rho}).
 \item
  when $ \theta_\ro ({H_\ro})\le H\le \varrho\tau_\mu'(0)$, it is reached at $u=  \theta_\ro ^{-1}(H)$ (with $\delta =\delta_\ro(H_\ro)$), and it equals $H\frac{ d(1-\ro) +\varrho\sigma_\mu(\theta_\ro ^{-1}(H))}{\varrho \theta_\ro ^{-1}(H)}= \sigma_\mu(\theta_\ro ^{-1}(H))$, where the last equality comes from \eqref{defthetarho}.
\end{itemize}
Putting the previous estimates together shows that $\widetilde D_{\mu,\ro,\eta}$ indeed coincides with $D_{\mu,\ro,\eta}$, given by  formula \eqref{spec1}.

\medskip

\noindent
{\bf Case 2.a.    $\sigma_\mu(H_{\min})\le  \frac{d(1-\varrho)}{1/\eta-\varrho}$ and  $H_{\ro,\eta}< {H_\ro}$:}  

\sk

Recall the definition  \eqref{defHmurho} of $H_{\ro,\eta}$, as well as the series of inequalities \eqref{ordre2}.

\mk

(i) When $ \ro\eta H_{\min} \leq H\le \varrho\eta H_{\ro,\eta}$, all those $u\in [H_{\min},\tau_\mu'(0)]$ such that $\varrho u/\delta\le H$ with $1\le \delta\le 1/\eta$ satisfy necessarily  $u\le   H /(\ro\eta) \leq H_{\ro,\eta}$.   Hence the maximum in \eqref{defdtilde} is attained when $u= H/\varrho\eta$ and $\delta=1/\eta$, i.e. for the largest possible $u$, and it equals $\sigma_\mu(H/\varrho\eta)$.

\mk

(ii) When  $H> \varrho\eta H_{\ro,\eta}$, remark that  all the values $\sigma_\mu(u)$ for $\varrho u/\delta\le H$, $1\le \delta\le 1/\eta$ and $u\le H_{\ro,\eta}$  are less than  $\sigma_\mu(H_{\ro,\eta})$, so they will not contribute to find the maximum in \eqref{defdtilde}. So, the discussion reduces to  {\bf Case 1.}, with the same conclusions.

\sk

This shows  that $\widetilde D_{\mu,\ro,\eta}$  and $D_{\mu,\ro,\eta}$ given by formula \eqref{spec2} coincide in this case.

\medskip

\medskip

\noindent
{\bf Case 2.b.    $\sigma_\mu(H_{\min})\le  \frac{d(1-\varrho)}{1/\eta-\varrho}$ and  $H_{\ro,\eta} \geq {H_\ro}$:}  

\sk

Recall  \eqref{ordre3}. The proof is essentially the same as the previous one, except that in this case, only $H_{\ro,\eta}$ matters:

(i) When $ \ro\eta H_{\min} \leq H\le \varrho\eta H_{\ro,\eta}$, by the same argument as above, all those $u\in [H_{\min},\tau_\mu'(0)]$ such that $\varrho u/\delta\le H$ with $1\le \delta\le 1/\eta$ satisfy necessarily  $u\le   H /(\ro\eta) \leq H_{\ro,\eta}$.   Hence the maximum in \eqref{defdtilde} is  $\sigma_\mu(H/\varrho\eta)$.

\mk

(ii) When  $H> \varrho\eta H_{\ro,\eta}$, the same argument as in {\bf Case 2.a (ii)} shows that  the maximal value is $\sigma_\mu(\theta_\ro ^{-1}(H))$.  
 \end{proof}

\subsection{Estimations of the local dimensions}

By analogy with Lemma \ref{lem3} and Proposition \ref{prop2}, the following holds.
\begin{lemma}
\label{lem3bis}
With probability 1:\begin{enumerate}
\item
  there exists a positive sequence $(\widetilde \ep_j)_{j\geq 1}$ such that for every $w\in \Sigma^*$,  
\begin{equation}
\label{majmin0bis}
 \mu(I_w^\ro) 2^{-|w|\widetilde \ep_{|w|} } \leq    \Mr(I_w)    \leq  \     \mu(I^\reta_w).
\end{equation}
\item for every $x\in \zu$, if $\delta_x=1$, then  $\dimi (\Mr,x))= \varrho \dimi (\mu,x)$.
\end{enumerate}
\end{lemma}
 
\begin{proof}
The proof is easily adapted from the ones of Lemma \ref{lem3} and Proposition \ref{prop2} by replacing $ I_w $ b y $ I_w^\ro$ everywhere it is necessary to (for instance,  in \eqref{adapt1} and then \eqref{majmin2} in Proposition \ref{prop2}).
\end{proof}
The elements of interest are now  those points in $\zu^d$ which are approximated at rate $\delta$ by dyadic numbers $x_w$ such that $\mu(I_w^{\ro\eta})$ (and not $\mu(I_w)$ as in Section \ref{sec-approx1}) has a specific value. This explains the introduction of the property $ {\mathcal {P}}(\rho, H,\widetilde \xi )$  below.

\begin{definition}
Let $H\geq 0$, $\de\geq 1$ and  $\widetilde \xi=( \xi_j)_{j\geq 1}$ a positive sequence.  The word $w\in \Sigma^*$ is said to satisfy  $\mathcal{P}(\ro,H,\widetilde \xi)$  whenever
  \begin{equation*}
   |I_w^{\varrho\eta}|^{H+\xi_{|w|}}   \leq    \mu( I_w^{\varrho\eta}) \leq  |I_w^{\varrho\eta}| ^{H-\xi_{|w|}}.
   \end{equation*}
Then one sets
 $$ \mathcal{A}_{\ro,H,\widetilde\xi,\delta} = \bigcap_{J \geq 1} \  \bigcup_{j \geq J} \ \bigcup_{w\in \mathcal{S}_{j}(\eta) :  \mbox{ {\small $w$ satisfies $\mathcal{P}(\ro,H,\widetilde \xi )$}   } }  B \Big (x_w, (\ell_w)^{\de}  \Big).$$
\end{definition}

The notations are consistent with those of the previous section: $\mathcal{P}(1,H,\widetilde \xi)=\mathcal{P}(H,\widetilde \xi)$ and $ \mathcal{A}_{1,H,\widetilde\xi,\delta} = \mathcal{A}_{H,\widetilde\xi,\delta} $.
Proposition~\ref{majexpo} extends to the following one.
\begin{proposition}
\label{majexpo'} With probability 1, if   $\widetilde\xi$ is a positive sequence converging to 0 and  $x\in\mathcal{A}_{\ro,H, \widetilde\xi,\delta}$,  then  
$$\ro\eta H \leq \dimi (\Mr,x) \leq \frac{\varrho H}{\min(\delta,1/\eta)}.$$
\end{proposition}
\begin{proof}
Again, the proof of Proposition \ref{majexpo} can be modified by replacing $I_w$ b y $I_w^\eta$, and using $\mathcal{P}(\ro,H,\widetilde \xi)$ instead of $\mathcal{P}(H,\widetilde \xi)$.
\end{proof}
To compute the multifractal spectrum of $\Mr$, the standard  inequality \eqref{formalism-1} cannot be used, since   the multifractal formalism fails.    We rather use the following properties of the sets $\underline E_{\Mr}(H)$.

\begin{proposition}\label{setestimates}
With probability 1, 
\begin{enumerate}
\item For every $H\ge 0$, writing $\tilde \ep$ for the constant sequence equal to $\ep$, we have 
$$
\underline E_{\Mr}^{\le}(H) \ \subset \  \bigcap_{\varepsilon>0} \  \bigcup_{\substack{H'\ge 0, \,  \delta\in [1,1/\eta], \,  \frac{\varrho H'}{\delta}\le h+\varepsilon}}     {\mathcal{A}}_{\ro, H',\tilde\varepsilon,\de}.
$$ 
\item For every $H\ge 0$, $\underline E_{\Mr}(H)\subset\overline E_\mu^{\ge}(H/\varrho)$.
\end{enumerate}
\end{proposition}

The first inclusion will be used to find an upper bound for the multifractal spectrum $\si_{\Mr}$ in the increasing part $H\leq \ro  \tau_
\mu'(0)$, while the second one will be used in the decreasing part $H\geq \ro  \tau_
\mu'(0)$.

\begin{proof}(1) Suppose that $x\in \underline E_{\Mr}(H)$. Recalling \eqref{equiv-dim}, there exists an increasing  sequence $(J_n)_{n\ge 1}$ of positive integers such that 
\begin{equation}
\label{converge1}
\lim_{n\to+\infty}  \frac{\log \wMr(I_{J_n}(x)) }{-J_n\log 2}=H.
\end{equation}
By Lemma~\ref{lem1},   there exists a  sequence of  positive integers $(J'_n)_{n\ge 1}$ satisfying  $J_n\le J'_n$ and $1\leq \limsup_{n\to \infty}J'_n/J_n\le 1/\eta$, such that for each $n$ there exists $w_n \in \mathcal{S}_{J'_n}(\eta)$ and $\wMr(I_{J_n}(x))=\mu(I_{w_n}^{\varrho\eta})$. Let us write $\mu(I_{w_n}^{\varrho\eta})=|I_{w_n}^{\varrho\eta}|^{ H_n}$ for some $H_n \in [\alpha_m,\alpha_M]$ (recall the uniform bound \eqref{uniform} for $\mu$). One has 
$$\frac{\log \wMr(I_{j_n}(x))}{-j_n\log 2} = \frac{\log \mu(I_{w_n}^{\varrho\eta})}{-J_n\log 2}  = \frac{- \lfloor J'_n\ro\eta\rfloor H_n}{-J_n} =\varrho H_n/\delta_n,$$
 with $\delta_n= \frac{ \ro J_n}{ \lfloor J'_n\ro\eta\rfloor } $.  By construction, 
 $$ 1 \leq \liminf _{n\to +\infty} \delta_n \leq \limsup_{n\to+\infty} \delta_n \leq 1/\eta.$$
Up to extraction of a subsequence, one may assume that the two sequences $(H_n)_{n\geq 1}$ and $(\delta_n)_{n\geq 1}$ converge respectively to $H' \in [H_{\min},H_{\max}]$ and $\delta\in [1,1/\eta]$. By \eqref{converge1}, one  has $\varrho H'/\delta=H$. To conclude, we remark that $|x-x_{w_n}|\le 2\cdot 2^{-J_n}  = 2\cdot 2^{-\de_n \lfloor \ro\eta J'_n\rfloor/\ro}$. So, recalling \eqref{deflw},   for every $\epsilon>0$, for every $n$ large enough so that  $|H_n-H'|\leq \ep/2$, $|\delta_n-\delta|\leq \ep/2$ and $|\varrho H_n/\delta-H|\leq \ep/2$:
\begin{itemize}
\item
${w_n}\in \mathcal{S}_{J'_n}(\eta)$, 
\item
$x\in B(x_{w_n}, 2\cdot 2^{-\de_n \lfloor \ro\eta J'_n\rfloor/\ro})\subset  B(x_{w_n},  \ell_{w_n}^{\delta-\epsilon})$ 
 \item
One has 
\begin{equation}
\label{minor}
\mu(I_{w_n}^{\varrho\eta})\in [|I_{w_n}^{\varrho\eta}|^{H'-\epsilon}, |I_{w_n}^{\varrho\eta}|^{H'+\epsilon}].
\end{equation}
 \end{itemize}
 
Finally, $x  $ belongs to the set $  {\mathcal{A}}_{\ro, H',\widetilde\epsilon,\de-\ep}$, hence the result.

\begin{remark}
One has $I_{w_n}\subset  \mathcal{N}(I_{J_n}(x))$. Let us make denote by $j_n$ the largest integer $j$ such that $I_{w_n} \subset  \mathcal{N}(I_{j}(x))$. Clearly $J_n\le j_n\le J'_n$ and $\wMr(I_{j_n}(x)) = \mu(I_{w_n}^{\varrho\eta})$. It follows that $\lim_{n\to\infty}\frac{J_n}{j_n}=1$, for otherwise $\dimi (\M_\ro,x)<H$. 
\end{remark}

\medskip

(2) With the same notations as above, for each $x\in \underline E_{\Mr}(H)$ we   deduce from \eqref{minor}  and the doubling property of $\mu$ that $\overline \dim(\mu,x)\ge H'=H\delta/\varrho\ge H/\varrho$ since $\delta \geq 1$. 
\end{proof}

\subsection{Useful bounds for estimating the Hausdorff dimensions of the sets $  {\mathcal{A}}_{\ro, H,\widetilde\xi,\de}$}

To get upper and lower  bounds for the multifractal spectrum $\si_{\Mr}$, we recall some results obtained in \cite{BSubiquity2}. There,  we studied the Hausdorff dimension of limsup sets of the form
$$\bigcap_{N\geq 1} \  \bigcup_{n\geq N:  \   (r_n)^{\ro H+\xi_{n}}   \leq    \mu( B(t_n, r_n^\ro )) \leq   (r_n)^{\ro H+\xi_{n}}  }  B(t_n, r_n ^\delta),$$
where  $H\geq 0$, $\delta\geq 1$, $(\xi_n)_{n\geq 1}$ is a positive sequence converging to zero, $\mu$ is a Gibbs measure, and where  the sequence of balls  $(B(t_n, r_n))_{n\geq 1}$ satisfies some properties described in the article.
This obviously looks exactly like our present situation with   a proper reformulation: $t_n$ and $r_n$ are replaced by $x_w$ and $\ell_w$, and the scaling condition on $ \mu( B(t_n, r_n^\ro ))$ is replaced by $\mathcal{P}(\rho,H,\widetilde \xi)$ (which does not change the result, as explained in Remark \ref{formequiv}).

In particular, a key property in  \cite{BSubiquity2} is called {\em weak redundancy}: the family of balls $(B(t_n, r_n))_{n\geq 1}$ is weakly redundant when calling $T_N= \{n: 2^{-N}\leq r_n\leq 2^{-N+1}\}$, there exists a sequence $(\alpha_N)_{N\geq1} $ of positive numbers such that
\begin{itemize}
\item
$\lim_{N\to +\infty} \frac{\alpha_N}{N} =0$.
\item
each $T_N$ can be decomposed into $\alpha_N$ families of pairwise disjoint  families $(T_{N,i})_{i=1,..., \alpha_N}$ such that for every $n\neq n'\in T_{N,i}$,  $B(t_n, r_n) \cap B(t_{n'}, r_{n'})=\emptyset$.
\end{itemize}
Weak redundancy of a family of balls implies that the balls of comparable size do not overlap too much (their overlaps are controlled by the sequence $(\alpha_N)_{N\ge 1}$). This is the case for the sequence $(B(x_w,\ell_w))_{w\in \mathcal{S}_j(\eta)}$ by   Lemma~\ref{lem2}. Hence, we can  use  results of \cite{BSubiquity2}, whose statements are recalled  now:

At first, we will need the following slight modification \cite[of Theorem 2.2]{BSubiquity2}, which readily follows from the proof of this theorem and the doubling property of $\mu$. 
\begin{theorem}With probability 1, for every Gibbs capacity $\mu$, every $H\geq 0$, $\delta\geq1$, and $\varepsilon>0$, denoting by $\tilde\varepsilon$ the constant sequence $(\ep)_{n\ge 1}$, 
\begin{equation}
\label{majrho}
\dim\,   {\mathcal{A}}_{\ro, H ,\tilde \ep,\de }\le \min \Big (\frac{d(1-\varrho)+\varrho S_\mu (H ,\varepsilon)}{\delta-\varepsilon}, S_\mu(H ,\varepsilon)\Big ), 
\end{equation}
where $S_\mu(H,\varepsilon)=\sup\{\sigma_\mu(H'): H'\in [H -\varepsilon,H+\varepsilon]\}$. 
\end{theorem}

Regarding the lower bound, since Lemma~\ref{lem1}  implies that with probability 1, not only \eqref{cover1} holds, but for all $j\ge 1$ large enough, for all $w\in \mathcal S_j(\eta)$, there exists $\widetilde S_{j,\varrho}(w)\subset \mathcal S_j(\eta)$ of cardinality at least $5^{-1}\cdot 2^{j(1-\varrho -\varepsilon_j)}$,  such that the balls $B(x_{w'}, \ell_{w'})$, $w'\in \widetilde S_{j,\varrho}(w)$, are included in $B(x_w,\ell^\varrho)$ and are pairwise disjoint,  
the following result holds (recall the formula \eqref{defmrho} for $m_\ro$):
\begin{theorem}[Corollary of Theorem  2.7 of \cite{BSubiquity2}] \label{ubirho}
Let $\mu$ be a Gibbs capacity.  Let $\varrho\in (0,1)$. 

With probability 1, for all $\de>1$, and $H\geq 0$ such that $\sigma_\mu(H)>0$,  there exists a sequence $\widetilde \xi$ decreasing to zero and a Borel probability measure $\nu_{\ro,H,\delta}$ such that:
\begin{itemize}
\item
$\nu_{\ro,H,\delta}( { \mathcal{A}}_{\varrho, H,\widetilde\xi,\de}) >0$,
\sk
\item
for any set $E$ of Hausdorff dimension less than $m_\ro(H,\delta) $,  $\nu_{\ro,H,\delta}(E)=0$.
\end{itemize}
In particular, $\dim {\mathcal{A}}_{\varrho, H,\widetilde\xi,\de}\ge m_\ro(H,\delta)$. 
\end{theorem}
\begin{remark}
Like for the case $\ro=1$, Theorem  2.7 of \cite{BSubiquity2} deals with $H\in(H_{\min},H_{\max})$, but a simple adaptation of the proof makes it possible to extend the result to $H\in \{H_{\min},H_{\max}\}$.
\end{remark}

\subsection{Upper bound for $\sigma_{\M_\ro}$}
\label{upper2}

To check the desired upper bound $\sigma_{\Mr}\le D_{\mu,\ro,\eta}$,  we distinguish the cases $H\le \varrho \tau_\mu'(0)$ and $H\ge \varrho \tau_\mu'(0)$. The first case  is a direct consequence of Proposition~\ref{pro5.1}, \eqref{majrho}, Proposition~\ref{setestimates}(1), and the following Lemma. 
\begin{lemma}
\label{propmaj3}For $\ep>0$, denote by $\tilde\ep$ the constant sequence equal to $\ep$.  
With probability 1, 
 for all $H\ge 0$, 
$$\dim  \Big( \bigcap_{\varepsilon>0}  \ \bigcup_{ {H' \ge 0,\,  \delta\in [1,1/\eta],\, \frac{\varrho H' }{\delta}\le H+\varepsilon}}     {\mathcal{A}}_{\ro, H'  ,\tilde \varepsilon,\de} \Big)\le \widetilde D_{\mu,\ro,\eta}(H).$$
\end{lemma}

\begin{proof} Fix $\ep\in (0,1)$. For every integer $N\ge 1/\varepsilon$, set $P_N=\{k/N:  k\in\N\}$. We have 
$$
\bigcup_{\substack{H' \ge 0, \, \delta\in [1,1/\eta],\,  \frac{\varrho H' }{\delta}\le H+\varepsilon}}     {\mathcal{A}}_{H' ,\varrho,\tilde\ep,\delta} =  \bigcup_{\substack{ H' \in P_N , \, \delta\in [1,1/\eta]\cap P_N, \,  \frac{\varrho H' }{\delta}\le H+\varepsilon}}     {\mathcal{A}}_{H' ,\varrho,\tilde\varepsilon,\delta}.
$$

By Lemma~\ref{lem2} the family of balls $(B(x_w,\ell_w))_{w\in\bigcup_{j\ge 1}\mathcal{S}_j(\eta)}$ satisfies  the weakly redundancy property \cite[Definition 2.1]{BSubiquity2}. Thus  by \eqref{majrho},
\begin{align*}
\dim \bigcap_{\ep>0}\bigcup_{\substack{H' \ge 0, \delta\in [1,1/\eta]\\ \frac{\varrho H' }{\delta}\le H+\ep}}     {\mathcal{A}}_{H' ,\varrho,\delta,\ep}&\le \inf_{\ep>0} \liminf_{N\ge 1/\ep} \sup_ {\substack{ H' \in P_N, \delta\in [1,1/\eta]\cap P_N\\ \frac{\varrho H' }{\delta}\le H+\ep}} \min \Big (\frac{d(1-\varrho)+\varrho S_\mu(H' ,\ep)}{\delta-\ep}, S_\mu(H' ,\ep)\Big )\\
&=\widetilde D_{\mu,\ro,\eta}(H),
\end{align*}
the conclusion coming from \eqref{defdtilde} and the continuity of the upper-semicontinuous concave function $\sigma_\mu$ over its compact domain $[H_{\min},H_{\max}]$. 
\end{proof}

For the second case $H\geq \ro\tau_\mu'(0)$, combining item (2) of  Proposition~\ref{setestimates} and \eqref{upcap2} applied to $\mu$ yields
$$\dim \underline E_{\Mr}(H) \leq \dim \overline{E}^{\ge} (H/\ro )\le \si_\mu(H/\ro)=  D_{\mu,\ro,\eta}(H).$$
This completes the proof.

\subsection{{Lower bound for the spectrum}}

Here one proves that $\si_{\Mr}\geq D_{\mu,\ro,\eta}$. The main idea is to find, for each $H$, with well-chosen parameters $H'$ and $\delta$, either a suitable measure of the form $\mu_{H'}$ as in Proposition~\ref{proplast}, or a measure $\nu_{\ro,H',\delta}$ from Theorem \ref{ubirho} carried by the set  $\underline E_{\Mr}(H)$.

\subsubsection{{\bf Decreasing part of the spectrum:  $H\ge \varrho\tau_\mu'(0)$}} 
If $H_{\min}\le H\le \ro H_{\max}$, by item (2) of  Lemma~\ref{lem3bis}, one knows that $ \{x\in[0,1]^d:\, \dim(\mu,x)=H/\ro \} \bigcap\{x\in[0,1]^d:\,\delta_x=1\}  \subset      E_{\Mr}(H) \subset   \underline  E_{\Mr}(H)$. This implies that the measure $\mu_{H/\ro}$  considered in Proposition~\ref{proplast} is supported on  $\underline E_{\Mr}(H))$, hence  $\dim \underline E_{\Mr}(H) \geq \si_\mu(H/\ro)$. If $H> \ro H_{\max}$, then $\sigma_{M_\ro}(H)\ge -\infty=D_{\mu,\ro,\eta}(H)$. 
 
\subsubsection{{\bf Increasing part of the spectrum when $\sigma_\mu(H_{\min})>  \frac{d(1-\varrho)}{1/\eta-\varrho}$ }}$\ $

\medskip

The three phases exhibited by formula~\ref{spec1} correspond to various choices for $H'$ and $\delta$ to get from Theorem~\ref{ubirho} a good measure $\nu_{\ro,H',\delta}$ 
supported on $\underline{E}_{\M_\ro}(H')$.

\mk

\noindent
{\it $\bullet$ $H\le  \varrho\eta H_\ro $.} Let $H'=H/\varrho\eta$ and  $\delta=1/\eta \ge \delta_\ro(H')$. 

Note that  $\sigma_\mu(H')> 0$ or $\sigma_\mu(H')=-\infty$. If $\sigma_\mu(H')> 0$, consider the measure  $\nu_{\ro,H',\delta}$ and the set ${\mathcal{A}}_{\varrho,H',\de,\widetilde\xi}$  from Theorem \ref{ubirho} (for an appropriate  sequence $\widetilde \xi$). By Proposition~\ref{majexpo'},  
\begin{equation}
\label{embed1}
 {\mathcal{A}}_{\varrho,H',\de,\widetilde\xi}\subset  \underline{E}_{\M_\rho}^\le(H).
 \end{equation}
Moreover, by construction, for all $n\ge 1$, we know using the upper bound for $\si_{\Mr}$ obtained in Section \ref{upper2} that   
$$\dim \underline{E}_{\M_\ro}^{\le}(H-1/n) \le \widetilde D_{\mu,\ro,\eta}(H-1/n)<D_{\mu,\ro,\eta}(H)=\frac{d(1-\varrho)+\varrho\sigma_\mu(H/\ro\eta ))}{\delta}=m_\ro(H,\delta),$$
the last inequality coming  from the formula \eqref{spec1} for $D_{\mu,\ro,\eta}$. Consequently, one has $\nu_{\ro,H',\delta}({\mathcal{A}}_{\varrho,H',\de,\widetilde\xi})=1$ and   $\nu_{\ro,H',\delta}\big(\underline{E}_{\M_\ro}^{\le}(H-1/n)\big) =0$ for every $n\geq 1$, hence  $\nu_{\ro,H',\delta}$  is supported  on the set
 $$ F= {\mathcal{A}}_{\varrho,H',\de,\widetilde\xi}\setminus \bigcup_{n\geq 1 }\underline{E}_{\M_\ro}^{\le}(H-1/n)  \subset   \underline E_{\Mr}(H).$$
This shows that   $\nu_{\ro,H',\delta}(\underline E_{\Mr}(H) ) = \nu_{\ro,H',\delta}( F) =1$, hence $\dim \underline E_{\Mr}(H) \ge   D_{\mu,\ro,\eta}(H)$. 

If $\sigma_\mu(H')=-\infty$, then  $\widetilde D_{\mu,\ro,\eta}(H)=D_{\mu,\ro,\eta}(H)=-\infty$, so $\underline E_{\Mr}(H)=\emptyset$ by ~\ref{setestimates}(1).

\medskip

\noindent
{\it $\bullet$ $ \varrho\eta H_\ro  \le H \le  \theta_\ro(H_\ro) $.} Here we set $H'=H_\ro $ and  $\delta=\delta_\ro (H_\varrho) \theta_\ro(H_\ro)  /H  = \ro H_\ro/H $ and fix $\widetilde \xi$ like in Theorem~\ref{ubirho}. 
Since $\frac{\ro H'}{\delta} =  \frac{\ro H_\ro}{\ro H_\ro/H}  = H$,  by Proposition~\ref{majexpo'},  
\eqref{embed1} still holds true, and one can imply the same arguments  as when $H\le  \varrho\eta H_\ro $. The lower bound for $\dim \underline E_{\Mr}(H) $ is then 
$$  m_\ro(H',\de) = \frac{d(1-\varrho)+\varrho\sigma_\mu(H') }{\delta}  = \frac{d(1-\varrho)+\varrho\sigma_\mu(H_\ro ) }{ \delta_\ro (H_\varrho) H_\ro /H}  = \frac{\si_\mu(H_\ro)}{\theta_\ro(H_\ro) } H $$
hence the desired value.

\medskip

\noindent
{\it $\bullet$ $ \theta_\ro(H_\ro)  \le H \le \varrho \tau_\mu'(0)$.} Here we set $H'=\theta_\ro ^{-1}(H)$, $\delta= \de_\ro(\theta_\ro ^{-1}(H))$ and fix $\widetilde \xi$ like in Theorem~\ref{ubirho}. The same argument as $H\le  \varrho\eta H_\ro $  yields the   lower bound
$$  m_\ro(H',\de) = \frac{d(1-\varrho)+\varrho\sigma_\mu(H') }{\delta}  = \frac{d(1-\varrho)+\varrho\sigma_\mu(\theta_\ro ^{-1}(H) ))}{\de_\ro(\theta_\ro ^{-1}(H)) }  =  \si_\mu(\theta_\ro^{-1}(H)) , $$
hence the expected value.

\subsubsection{{\bf Increasing part of the spectrum when $\sigma_\mu(H_{\min})>  \frac{d(1-\varrho)}{1/\eta-\varrho}$ and $H_{\ro,\eta}\le H_{\varrho}$}}$\ $

We only need to deal with the case $H\le \varrho\eta H_{\ro,\eta}$, since the other cases can be treated as in the previous subsection. 

\medskip

If $H\le \varrho\eta H_{\ro,\eta}$ and $\sigma_\mu(H/\varrho\eta)\ge 0$, set $H'= H/\varrho\eta$ and $\delta=1/\eta \le \delta_\ro (H')$. If $\sigma_\mu(H')\ge 0$ and $\sigma_\mu(H')=0$, then $H'=H_{\min}$. In this case, note that by item (1) of Lemma~\ref{lem3bis} one has $\underline\dim(\M_\ro,x)\ge \ro\eta H_{\min}$ for all $x\in[0,1]^d$. Moreover,  it is not difficult to show that there exists a sequence $\widetilde\xi$ such that $ {\mathcal{A}}_{\varrho,H_{\min},1/\eta,\widetilde\xi}\neq\emptyset$. Since by Proposition~\ref{majexpo'}, ${\mathcal{A}}_{\varrho,H_{\min},1/\eta,\widetilde\xi}\subset \underline{E}_{\M_\rho}^\le(\ro\eta H_{\min})$, this yields $\underline{E}_{\M_\ro}(\ro\eta H_{\min})\neq\emptyset$ and $\sigma_{\M_\ro}(\ro\eta H_{\min})\ge 0=\sigma_\mu(H_{\min})=D_{\mu,\ro,\eta}(\ro\eta H_{\min})$. 

If $\sigma_\mu(H')>0$, the discussion is the same as in the previous  section except that here   $m_\ro(H',\delta)=\sigma_\mu(H')$ instead of $m_\ro(H',\delta) =(1-\varrho+\varrho\sigma_\mu(H'))/\delta$. 

If $\sigma_\mu(H')=-\infty$, then, as in the previous subsection, $\widetilde D_{\mu,\ro,\eta}(H)=D_{\mu,\ro,\eta}(H)=-\infty$, so $\underline E_{\Mr}(H)=\emptyset$ by Proposition~\ref{setestimates}(1).

\subsubsection{{\bf Increasing part of the spectrum when  $\sigma_\mu(H_{\min})>  \frac{d(1-\varrho)}{1/\eta-\varrho}$ and $H_{\ro,\eta}> H_{\varrho}$}}$\ $

\medskip

\noindent
{\it $\bullet$ $ H\le  \varrho\eta H_{\ro,\eta}$.} The discussion is the same as in the previous subsection.

\medskip

\noindent
{\it $\bullet$ $ \varrho\eta H_{\ro,\eta}\le H\le \varrho\tau_\mu'(0)$.} The discussion is identical as in the previous subsection when $H\ge H_\ro$, since we know from Section~\ref{sec5-1} that  $\varrho\eta H_{\ro,\eta}\ge H_\ro$.

\subsection{About the  multifractal formalism for $\Mr$}

In this section, we compute $\tau_{\Mr}$ and study the (non)-validity of the multifractal formalism for $\Mr$. First we restate item (2) of Theorem \ref{main2}.

\begin{proposition}\label{tauMrho}
With probability 1, one has  
$$
\tau_{\M_{\ro}} (q) =
\begin{cases}
d(\varrho-1)+\varrho\tau_\mu(q)&\text{if }q<{q_\ro},\\
\eta(d(\varrho-1)+\varrho\tau_\mu(q))&\text{if }q\ge {q_\ro}.
\end{cases}
$$
\end{proposition}

\begin{proof}

For $j\geq 1$, one estimates the sum
$ \sum_{w\in \Sigma_j} \wMr( I_w)^q$.

\mk

 {\em (i)  Lower bound for $\tau_{\Mr}$:}

\mk
 {\bf $\bullet$  Case   $q<0$:}  using the lower bound in item (1) of Lemma \ref{lem3bis}, one has
\begin{eqnarray*}
 \sum_{w\in \Sigma_j} \wMr(  I_w)^q &  \leq     \sum_{w\in \Sigma_j} \max(\Mr(I): I\in \mathcal{N}(I_j(x)))^q  \leq  3^{d} \sum_{w\in \Sigma_j}   \mu(I^\ro_w)^{q} 2^{-q|w|\widetilde \ep_{|w|}}.
 \end{eqnarray*}
 
 In the last sum, each interval $I_{w'}$, for $w'\in \Sigma_{\lfloor j\ro \rfloor}$, appears as $I_w^\ro$ in $2^{d(j-{\lfloor j\ro\rfloor})}$ terms. So
 \begin{eqnarray*}
 \sum_{w\in \Sigma_j} \wMr(  I_w)^q &  \leq   C 2^{d(j-{\lfloor j\ro \rfloor})} \sum_{w' \in \Sigma_{\lfloor j\ro \rfloor} }   \mu(I_{w'})^{q} 2^{-qj\widetilde \ep_{j}} =  C 2^{d(j-{\lfloor j\ro \rfloor})}  2^{-{\lfloor j\ro \rfloor}\tau_{\mu,j}(q)}.
 \end{eqnarray*}
Letting $j$ go to infinity yields  $\tau_{\Mr}(q) \geq  d(\ro-1) + \ro \tau_{\mu}(q)$.  

\mk

\sk {\bf $\bullet$  Case $q>0$:}      The same arguments as in Proposition \ref{propupper} can be adapted.   By \ref{lem1},  for every word $w\in \Sigma_j$, the value of $\wMr(I_w)$ is reached for one word $w' $ whose length is between $j$ and $j/(\eta-\ep_j)$, and  each word $w'  \in \mathcal{S}_{j'}(\eta)$ with $j  \leq j' \leq j/(\eta-\ep_j)$ may contribute to at most $3^{d}$ values of $\wMr(I_w)$ for  $|w|=j$. Hence   the upper bound in item  (1) of Lemma \ref{lem3bis} gives 
\begin{eqnarray*}
 \sum_{w\in \Sigma_j} \ \wMr (I_w)^q &  \leq  & 3^{d}  \sum_{j'=j}^{\lfloor j/(\eta-\ep_j)\rfloor } \sum_{w\in \mathcal{S}_{j'}(\eta) }  \mu(I_w^{\ro\eta} )^q  \leq 3^{d}   \sum_{j'=j}^{\lfloor j/(\eta-\ep_j)\rfloor } \sum_{w\in\Sigma_{j'}}   p_w\mu(I^{\ro\eta} _w) ^{q}
   \end{eqnarray*}
Taking expectation, one gets  
$$
\mathbb{E} \Big ( \sum_{w\in \Sigma_j} \  \wMr (I_w)^q\Big )\le 3^{d}   \sum_{j'=j}^{\lfloor j/(\eta-\ep_j)\rfloor } \sum_{w\in\Sigma_{j'}}   2^{-dj'(1-\eta)}\mu(I^{\ro\eta} _w) ^{q}.
$$
As before, each cube $I_{w'}$, for $w'\in \Sigma_{\lfloor {\ro\eta}  j'\rfloor }$ and $j  \leq j' \leq j/(\eta-\ep_j)$, contains  at most  $3^{d}   \times 2^{d(j'-\lfloor j'{\ro\eta}) \rfloor }$  dyadic cubes   of generation $2^{j'}$. Thus,
\begin{eqnarray}
\nonumber 
 \mathbb{E} \Big ( \sum_{w\in \Sigma_j}  \wMr (I_w)^q\Big )&  \leq  &    3^d   \sum_{j'=j}^{\lfloor j/(\eta-\ep_j)\rfloor } \sum_{w\in \Sigma_{[j'{\ro\eta} ]}}   2^{  d(j'-\lfloor j'{\ro\eta}) \rfloor }    2^{-dj'(1-\eta)} \mu(I_w)^{q} \\
\label{malE1ro}  & \leq &  C \sum_{j'=j}^{\lfloor j/(\eta-\ep_j)\rfloor }   2^{d \eta (1-\ro) j'}  2^{- {\lfloor j' \ro\eta \rfloor }\tau_{\mu,{[j'\ro\eta]}}(q)  }   .
 \end{eqnarray}
Again, the behavior of the last sum depends on the value of $\tau_{\mu}(q) $. 

\sk

\sk {\bf $\bullet$  $0<q<q_\rho$:} since by \eqref{deftaurho} one has  $d(\ro-1)+\ro\tau_\mu(q_\ro)=0$, one sees that $d(\ro-1)+\ro\tau_\mu(q)<0$.  Fixing  $s<d(\ro-1)+\ro\tau_\mu(q)   $ and $\ep = (d(\ro-1)+\ro\tau_\mu(q) -s)/3$.  Recalling \eqref{lim-tau}, for $j'$ sufficiently large, $|\tau_{\mu,[j'\ro\eta]}(q)  - \tau_\mu(q)|\leq \ep$, so that 
 \begin{eqnarray*}
 \mathbb{E} \Big ( \sum_{w\in \Sigma_j}  \wMr (I_w)^q\Big )    & \leq  & C   \sum_{j'=j}^{\lfloor j/(\eta-\ep_j)\rfloor } 2^{- j'\eta ( d(\ro-1)+\ro\tau_\mu(q) )- \ep)  } \leq C 2^{-  j\eta/(\eta-\ep_j) ((d(\ro-1)+\ro\tau_\mu(q) ) - \ep)  } .
\end{eqnarray*}
Then,  as before,
$$\mathbb{E} \Big (\sum_{j\ge 1} 2^{js} \sum_{w\in \Sigma_j} \ \wMr(I_w)^q\Big ) \leq  C \sum_{j\ge 1} 2^{j(s-  \eta/(\eta-\ep_j) (d(\ro-1)+\ro\tau_\mu(q)  )- \ep)   ) }  <\infty,$$ 
from which we deduce that, almost surely,
$$\sum_{j\ge 1} 2^{js} \sum_{w\in \Sigma_j} \ \wMr1(I_w)^q <\infty.$$
 Consequently, $\tau_{\Mr}(q) \geq  s$, and letting $s$ tend to $d(\ro-1)+\ro\tau_\mu(q) $ yields the result.

\sk

\sk {\bf $\bullet$  $q>q_\ro$:}    Fix  $0< s<d(\ro-1)+\ro\tau_\mu(q)  $,  $\ep = (d(\ro-1)+\ro\tau_\mu(q) -s)/3$ and $j'$ large.  Since $d(\ro-1)+\ro\tau_\mu(q)  >0$, the sum in \eqref{malE1ro}  is bounded above by the therm $j'=j$, and   
 $$
 \mathbb{E} \Big ( \sum_{w\in \Sigma_j}  \wMr (I_w)^q\Big )     \leq C 2^{-  j\eta (d(\ro-1)+\ro\tau_\mu(q) - \ep)  } .
$$

 The same conclusion as above gives   $\tau_{\Mr}(q) \geq  \eta (d(\ro-1)+\ro\tau_\mu(q) ) $, almost surely.

\sk

 The case where $\tau_{\mu}(q) =0$ is obtained by continuity.  

\mk

{\em (ii) Upper bound for $\tau_{\Mr}$:}
 \mk
 
$\bullet$  Case  {\bf   $q\ge 0$:} We show that with probability 1, for all $q\ge 0$,
$$
\tau_{\Mr}(q)\le \min (\eta(d(\varrho-1)+ \varrho\tau_\mu(q)), d(\varrho-1)+ \varrho\tau_\mu(q))
=\begin{cases}
\eta(d(\varrho-1)+ \varrho\tau_\mu(q))&\text{if }q\le q_\varrho\\
d(\varrho-1)+ \varrho\tau_\mu(q)&\text{if }q\ge q_\varrho.
\end{cases}
$$
By Lemma~\ref{lem3bis},  with probability 1, $\Mr(I_w)\ge \mu(I_w^\varrho)2^{-j\widetilde \epsilon_j}$ for all $w\in \Sigma_j$, $j\ge 1$. Consequently, for every $q\ge 0$
\begin{align*}
\sum_{w\in \Sigma_j} \wMr(I_w)^q \ge \sum_{w\in \Sigma_j} \mu(I_w^\varrho)^q2^{-qj\widetilde \epsilon_j}.
\end{align*}
By the arguments already used before, each $I_{w'}$, for $w'\in \Sigma_{\lfloor \ro j\rfloor}$, appears $2^{d(j- \lfloor \ro j\rfloor)}$ times in the above sum. Hence
\begin{align*}
\sum_{w\in \Sigma_j} \wMr(I_w)^q \ge  \sum_{w'\in \Sigma_{\lfloor \varrho j\rfloor }} \mu(I_{w'})^q 2^{d(j-\lfloor \varrho j\rfloor)}2^{-qj\widetilde \epsilon_j}  = 2^{-\lfloor \varrho j\rfloor\tau_{\mu,\lfloor \varrho j\rfloor }(q) +d(j-\lfloor \varrho j\rfloor) - qj\widetilde \epsilon_j} .
\end{align*}
Letting $j$ tend to infinity yields  $\tau_{\Mr}(q)\le d(\varrho-1)+ \varrho\tau_\mu(q)$ for all $q\ge 0$.

\mk
 
 To get the other bound, let $H=\tau_\mu'(q)$. Consider the measure $\mu_H$ provided by Proposition~\ref{proplast} as in Lemma~\ref{lem1}. Since $\mu_H$ is exact dimensional and supported on $E_\mu(H)$, and $\mu$ is doubling, given $\ep>0$, for $J$ large enough, there is a subset $\Sigma_J(H,\ep)$ of $\Sigma_J$ of cardinality at least $2^{J(\tau_\mu^*(q)-\ep)}$ such that for all $u\in  \Sigma_J(H,\ep)$ one has $2^{-J(H+\ep)}\le \mu(I_u)\le 2^{-J(H-\ep)}$. 
 
 We now use  Lemma~\ref{lem1} and the sequence   $(\ep_j)_{j\ge 1}$ defined therein (which does not depend on $q$). With probability 1, when $j$ is large enough, every word $v \in \Sigma_{\lfloor(\eta-\ep_j)j\rfloor}$  contains at least one  $w_v\in\Sigma_j$ such that  $p_{w_v}=1$. Hence, for $j$ large enough, taking $J=\lfloor \varrho \eta j\rfloor$, fixing $u\in \Sigma_J(H,\ep)$, this applies to each element $v \in \Sigma_{\lfloor(\eta-\ep_j)j\rfloor}$ such that $I_v\subset I_u$: $I_v$  contains a subcube   $I_{w_v}$ with   $w_v\in\Sigma_j$ and  $p_{w_v}=1$. 
 In particular, $\wMr(I_{w_v})\ge \mu(I_{w_v}^{\ro\eta})$. Also, $I_{w_v}^{\ro\eta} \in \Sigma_J$, so  $I_{w_v}^{\ro\eta} = I_u$,  and $\wMr(I_{v})\ge \mu(I_u)\geq 2^{-J(H+\ep)}= 2^{-\ro \eta j (H+\ep)}$.
 
 There are $2^{d(\lfloor(\eta-\ep_j)j\rfloor-J)} = 2^{d(\lfloor(\eta-\ep_j)j\rfloor-\lfloor \ro\eta j\rfloor )}$ such disjoint  subcubes $I_v$. These observations and the fact that $q\ge 0$ imply that 
\begin{align*}
 \sum_{w\in \Sigma_j} \wMr(I_w)^q  & \ge  \sum_{u\in \Sigma_J(H,\ep)} \ \  \sum_{v \in \Sigma_{\lfloor(\eta-\ep_j)j\rfloor}: I_v\subset I_u}   \ \ \sum_{w \in \Sigma_j : I_{w}\subset I_v \mbox {and } p_{w_v}=1} \wMr(I_{w})^q\\
 & \ge  2^{\lfloor \ro\eta j\rfloor (\tau_\mu^*(q)-\ep)}  2^{d(\lfloor(\eta-\ep_j)j\rfloor-\lfloor \ro\eta j\rfloor )} 2^{-\ro \eta j (H+\ep)}.
 \end{align*}
 This yields  $\tau_{\Mr}(q)\le \eta(d(\ro-1)+\varrho (qH-\tau_\mu^*(q)))+O(\ep)$. Since  $qH-\tau_\mu^*(q)=\tau_\mu(q)$ and $\ep$ is arbitrary, we conclude that with probability 1, $\tau_{\Mr}(q)\le  \eta(d(\ro-1)+\varrho\tau_\mu(q))$ for all $q\ge 0$.
 \mk

 {\bf $\bullet$  Case   $q<0$:}  Let us make two preliminary observations. 
 
 First,  if $1<\gamma<1/\eta$, for all $j\ge 1$ and $w\in \Sigma_j$,  consider the event  $\mathcal E(I_w)=\{\exists\,  v\in \bigcup_{k=j}^{\lfloor \gamma j\rfloor}\Sigma_k: I_v\subset 3I_w\text{ and }p_v=1\}$. When $\mathcal E(I_w)$ does not hold for some $w\in \Sigma_j$, all the $\mu(I_v)$  are put to zero by the sampling process, for $I_v\subset3 I_w$ and $|v| \in \{j+1,..., \gamma j\}$.  So intuitively, the corresponding surviving weight $\wMr(I_w)$ will be small. By independence of the $p_v$'s,  
 \begin{align}
\label{proba-ew}
 \mathbb P(\mathcal E(I_w))=1-\Big (\prod_{k=j}^{\lfloor \gamma j\rfloor} (1-2^{-d(1-\eta)k})^{2^{d(k-j)}}\Big )^{3^d}\le 1- \exp(-C 2^{-dj(1-\gamma\eta)})\le C 2^{-dj(1-\gamma\eta)}
 \end{align}
for some $C>0$ depending on $(d, \eta, \gamma)$.

 Next, we are interested in counting how many times $\mathcal E(I_w)$ does not hold for $I_w \in \Sigma_{j'} $ inside a given cube $I \in \Sigma_J$ with $J\leq {j'}$. For this,  fix   $\ep\in (0,1)$ and $\theta>1$, as well as a Borel probability measure $\nu$ on $[0,1]^d$. It is easily checked that by combining \eqref{proba-ew}, the Markov inequality and the Fubini-Tonelli Theorem,  
 \begin{align}
\nonumber &
\sum_{J\ge 1} \ \sum_{{j'} \ge \theta J}\mathbb{E}\left(\nu\left(\left\{x\in[0,1]^d: \ \mathcal{L}^d(\{y\in I_J(x): \mathcal E(I_{j'} (y))\text{ holds }\})\ge 2^{-dJ} (1-2^{-d({j'}-J)\ep})\right\}\right)\right )\\
&\le \sum_{J\ge 1} \ \sum_{{j'}\ge \theta J}(1-2^{-d({j'}-J)\ep})^{-1} C 2^{-d{j'}(1-\gamma\eta)}<\infty.
\label{eq-fin2}
\end{align}
Now, for $q<0$, let $H=\tau_\mu'(q)$ and take for $\nu$ the measure $\mu_H$ of Proposition~\ref{proplast}.  By the Borel-Cantelli Lemma and \eqref{eq-fin2}, with probability 1, for $\mu_H$-almost every $x$, for $J$ large enough, for all $j\ge \theta J$,  
$$\mathcal{L}^d(\{y\in I_J(x): \mathcal E(I_{j'}(y))\text{ does not hold }\})\ge 2^{-dJ}2^{-d({j'}-J)\ep}=2^{d({j'}-J)(1-\ep)}2^{-d{j'}}.$$
Hence for every $j'\geq \theta J$,  there exists a subset $\Sigma_{j'}(I_J(x)))$ of $\Sigma_{j'}$ of cardinality at least $2^{\lfloor d ({j'}-J)(1-\ep)\rfloor}$ such that for all $w\in \Sigma_{j'}(I_J(x))$, $I_w\subset I_J(x)$ and $\mathcal{E}(I_w)$ does not hold.  
 
 Take $\theta= (1+\ro^{-1})/2$, $\ep\in(0,1\land (\eta^{-1}-1)/2)$, and $\gamma=1/\eta-\ep$.  Remark that with this choice of parameter, if $J=\lfloor \ro j\rfloor$, then   $j \geq \ro^{-1} J >\theta j$, so the previous conclusion holds for $j'=j$.

 Combining what just precedes with the properties of $\mu_H$ already used in the case $q\ge 0$, one obtains that for all $J$ large enough of the form $J=\lfloor \ro j\rfloor$, there exists:
 \begin{itemize}
 \item a subset $\Sigma_J(H,\ep)$ of $\Sigma_J$ of cardinality at least $2^{J(\tau_\mu^*(q)-\ep)}$ such that for all $u\in  \Sigma_J(H,\ep)$ one has $\mu(I_u)\le 2^{-J(H-\ep)}$, 
 \item
 for all $u\in  \Sigma_J(H,\ep)$, there is    a subset $\Sigma_j(I_u)$ of $\Sigma_j$ of cardinality at least $2^{\lfloor d(j-J)(1-\ep)\rfloor}$ such that for all $w\in \Sigma_j(I_u)$, $I_w\subset I_u$ and $\mathcal{E}(I_w)$ does not hold. 
   \end{itemize}
   
   Hence for each $w\in \Sigma_j(I_u)$,  the dyadic cubes $I_v$ contained in $3I_w$ such that $p_v=1$ are of generation at least $\lfloor(\eta^{-1}-\ep)j\rfloor$. For such surviving vertices $v$, their generation $g(v)$ satisfies $ \lfloor \ro\eta  g(v)\rfloor \geq \ro  \eta  \lfloor(\eta^{-1}-\ep)j\rfloor  \geq  j(\ro  -C'\ep) $,   for some constant $C'$ depending on  $\eta$ and $\mu$. So  $I^{\rho\eta}_v \subset 
 I_u^{  (\rho-C'\ep)/\ro}$, and the doubling property \eqref{quasib} of $\mu$  together with \eqref{uniform} imply that    $\mu(I^{\rho\eta}_v)\le \mu(I_u)2^{C''j\ep}$ for some constant $C''$.  Hence $\widetilde \Mr (I_w)\le \mu(I_u)2^{C'j\ep}\le 2^{-J(H-\ep)} 2^{C''j\ep}$.

 Putting everything together,  in the case where $q<0$, one finally gets  
  \begin{align}
\nonumber 
 \sum_{w\in \Sigma_j} \wMr(I_w)^q & \ge   \sum_{u\in \Sigma_J(H,\ep)} \ \  \sum_{w  \in \Sigma_j(I_u) }   \ \   \wMr(I_{w})^q\\    
 & \geq 2^{J(\tau_\mu^*(q)-\ep)} 2^{\lfloor d(j-J)(1-\ep)\rfloor}(2^{-J(H-\ep)} 2^{C''j\ep})^q.
\end{align}
Recalling that $J=\lfloor \ro j\rfloor$,  one deduces that $\tau_{\Mr}(q)\le d(\ro-1)+\ro(qH-\tau_\mu^*(H))+O(\ep)=d(\ro-1)+\ro\tau_\mu(q)+O(\ep)$. Since $\ep$ is arbitrary, this yields the desired upper bound, for all $q<0$, almost surely.

 \sk
 
 Finally, $\tau_{\Mr}$ and $\tau_\mu$ being continuous functions, all the previous inequalities relating these functions hold almost surely for all    $q<0$.
 \end{proof}

Next  proposition  entails the failure of the multifractal formalism over    subsets of~$\mathrm{dom}(\sigma_{\Mr})$. 

\begin{proposition}
Recall the definifion of $I_{\ro,\eta}$ in Theorem~\ref{main2}. With probability 1, the multifractal formalism holds over $I_{\ro,\eta}$ and it fails over $\mathrm{dom}(\sigma_{\M_\ro})\setminus I_{\ro,\eta}$. 
\end{proposition}

\begin{proof}
It follows from Proposition~\ref{tauMrho} that, with probability 1, 
\begin{equation}
\label{taumuret}
\tau_{\Mr}^*(H)= \begin{cases}
\eta d (1-\ro) +\eta \ro \si_\mu(H/\ro\eta)& \mbox{ when }   H\leq \ro\eta H_\ro,\\
\ \ \  \ \ \ \ \ \frac{ \si_\mu(H_\ro)}{\theta_\ro(H_\ro)}H& \mbox{ when } \ro\eta H_\ro <H\leq \ro H_\ro,\\
\ \ \ d(1-\varrho)+\varrho\sigma_\mu(H/\varrho) & \mbox{ when } H > \ro H_\ro .
\end{cases}
\end{equation}

\mk

{\bf Case 1.}
One compares \eqref{taumuret} with \eqref{spec1}. 

The two formulas coincide when $H_{\min}\leq H\leq \theta_\ro(H_\ro)$, and at $H=\ro\tau_\mu'(0)$.

When $  \theta_\ro(H_\ro) <H \leq \ro H_\ro$,  $\tau_{\Mr}^*$ stays linear but $\si_{\Mr}$ becomes stricly concave, so $\si_{\Mr} < \tau_{\Mr}^*$.

When  $   \ro H_\ro  < H \leq \ro \tau_\mu'(0)$,     $ H_\ro\leq H /\ro \leq  \theta_\ro^{-1}(H)$, and since the mapping \eqref{funcro} is decreasing on $[H_\ro,H_{\max}]$, one has 
\begin{eqnarray*}
 \si_\mu(\theta_\ro^{-1} (H) )  & = &  \frac{d(1-\ro)+\ro \si_\mu(\theta_\ro^{-1} (H) )}{\delta_\ro (\theta_\ro^{-1} (H)) }=  \frac{ \theta_\ro^{-1} (H) }{\delta_\ro (\theta_\ro^{-1} (H)) } \frac{d(1-\ro)+\ro \si_\mu(\theta_\ro^{-1} (H) )}{\theta_\ro^{-1} (H) } \\
   & <&  \frac{ \theta_\ro^{-1} (H) }{\delta_\ro (\theta_\ro^{-1} (H)) }    \frac{d(1-\ro)+\ro \si_\mu(H/\ro)}{H/ \ro} .
   \end{eqnarray*}
   But  by definition of $\delta_\ro$, $ \frac{ \theta_\ro^{-1} (H) }{\delta_\ro (\theta_\ro^{-1} (H)) }    = \frac{ \theta_\ro(  \theta_\ro^{-1} (H))}{\ro} = H/\ro$, so $ \si_\mu(\theta_\ro^{-1} (H) )  <d(1-\ro)+\ro \si_\mu(H/\ro) $.

When $H> \ro \tau_\mu'(0)$, by convexity $\si_\mu(H/\ro) < d(1-\ro)+\ro\si_\mu(H/\ro)$.

\mk

{\bf Case 2.a.}
The formulas \eqref{taumuret} with \eqref{spec2} do not coincide any more when $H_{\min}\leq H <  \ro\eta H_{\ro\eta}$, and the other regions are similar to Case {\bf 1}. 

\mk

{\bf Case 2.b.}  This case is simpler and follows from the previous ones. 
 \end{proof}


\bibliographystyle{plain}
\bibliography{Biblio_BME}

\end{document}